\documentclass[12pt, notitlepage]{amsart}
\usepackage{latexsym, amsfonts, amsmath, amssymb, amsthm, cite}
\usepackage{graphicx}
\usepackage{mathrsfs}
\usepackage[linktocpage=true,colorlinks,citecolor=blue,linkcolor=blue,urlcolor=blue]{hyperref}
\usepackage{amssymb}
\usepackage{cite}
\usepackage{bbm}
\usepackage{amsmath}
\usepackage{latexsym}
\usepackage{amscd}
\usepackage{tikz}
\usepackage{amsthm}
\usepackage{mathrsfs}
\usepackage{url}
\usepackage[utf8]{inputenc}
\usepackage[english]{babel}
\usepackage{amsfonts}
\usepackage{mathtools}
\usepackage{comment}
\usepackage[colorinlistoftodos]{todonotes}

\numberwithin{equation}{section}
\def\Z{\mathbb Z}

\def\P{\mathbb P}

\def\E{\mathbb E}

\def\p {\mathbbm{1}_{\mathcal P}}

\def\CA{\mathcal{A}}
\def\CS{\mathcal{S}}

\newtheorem{theorem}{Theorem}[section]
\newtheorem{lemma}[theorem]{Lemma}

\newtheorem{proposition}[theorem]{Proposition}

\newtheorem{corollary}[theorem]{Corollary}

\theoremstyle{remark}
\newtheorem{remark}[theorem]{Remark}

\theoremstyle{definition}

\theoremstyle{remark}

\numberwithin{equation}{section}

\begin{document}
	\title{Central limit theorems for random multiplicative functions}
	\author{Kannan Soundararajan}
	\address{Department of Mathematics, Stanford University, Stanford, CA, USA}
	\email{ksound@stanford.edu}
	\author{Max Wenqiang Xu} \address{Department of Mathematics, Stanford University, Stanford, CA, USA}
	\email{maxxu@stanford.edu}
	\dedicatory{To Peter Sarnak on the occasion of his seventieth birthday}
	
\begin{abstract}  A Steinhaus random multiplicative function $f$ is a completely multiplicative function obtained by setting its values on primes $f(p)$ to be independent random variables distributed uniformly on the unit circle.  Recent work of Harper shows that $\sum_{n\le N} f(n)$ exhibits ``more than square-root cancellation," and in particular 
$\frac 1{\sqrt{N}} \sum_{n\le N} f(n)$ does not have a (complex) Gaussian distribution.  This paper studies $\sum_{n\in {\mathcal A}} f(n)$, where ${\mathcal A}$ is a 
subset of the integers in $[1,N]$, and produces several new examples of sets ${\mathcal A}$ where a central limit theorem can be established.  We also consider more general sums such as $\sum_{n\le N} f(n) e^{2\pi i n\theta}$, where we show that a central limit theorem holds for any irrational $\theta$ that does not have extremely good Diophantine approximations. 
\end{abstract}
\maketitle

\section{Introduction}

In recent years there has been a lot of progress in understanding the behavior of random multiplicative functions.  One motivation for studying such functions is that understanding these may help shed light on functions of interest in number theory such as Dirichlet characters or the Liouville and M{\" o}bius functions.  Two natural models for random multiplicative functions are (1) the Steinhaus model of a random completely multiplicative function $f(n)$ where the values $f(p)$ (on primes $p$) are chosen independently and uniformly from the unit circle, and (2) the Rademacher model of a multiplicative function $f(n)$ taking values $\pm 1$ on square-free integers and $0$ on integers having a square factor, with $f(p)$ chosen independently and uniformly from $\{ -1, 1\}$.   

 A fundamental question is to understand the distribution of the partial sums $\sum_{n\le N} f(n)$ for random multiplicative functions $f$ (either in the Steinhaus case or in the Rademacher case).  Since the values of $f$ at integers satisfy dependency relations, it is a challenging problem to understand this distribution.   A breakthrough result of Harper \cite{HarperLow} established that typically $\sum_{n\le N} f(n)$ is $o(\sqrt{N})$.  Note that $\sqrt{N}$ is the size of the standard deviation of $\sum_{n\le N} f(n)$, and thus Harper's result (which confirmed a conjecture of Helson  \cite{Helson}) exhibits ``more than square-root cancellation" in such partial sums.   

One of our goals in this paper is to explore the distribution of partial sums of random multiplicative functions when restricted to subsets ${\mathcal A}$ of $[1,N]$.  We shall give criteria and several examples of sets ${\mathcal A}$ where such partial sums satisfy a central limit theorem.  For simplicity, we describe our results in the Steinhaus setting, and sketch briefly (in Section~\ref{sec: rad}) corresponding results in the Rademacher case.  We begin with a simplified criterion for ${\mathcal A}$ where a central limit theorem holds (see Theorem~\ref{thm: a_n} for a more precise, but more technical, result).

\begin{theorem}\label{thm1.1}  
Let $N$ be large, and let $\mathcal{A}$ be a subset of $[1,N]$ with size 
\begin{equation} 
\label{1.1} 
|\CA|\ge N \exp(-\tfrac 13\sqrt{\log N \log \log N}).
\end{equation} 
Suppose that there exists a subset $\mathcal{S}\subset \mathcal{A}$ with size $|\mathcal{S}|=(1+o(1))|\mathcal{A}|$ 
satisfying the following criterion: 
\begin{equation} 
\label{1.2} 
\#\{ (s_1, s_2, s_3, s_4) \in {\mathcal S}^4: \ \ s_1 s_2 = s_3 s_4 \} = (2+o(1))|\mathcal{S}|^{2}. 
\end{equation} 
Then, as $f$ ranges over random multiplicative functions in the Steinhaus model, the quantity  
$$ 
\frac{1}{\sqrt{|\mathcal{A}|} } \sum_{n \in \mathcal{A}} f(n)
$$ 
is distributed like a standard complex normal random variable with mean $0$ and variance $1$. 
\end{theorem}

Here and below, when we say ``distributed like" we mean convergence in distribution as the parameter $N$ tends to infinity.
Recall that a real normal random variable $W$ with mean $0$ and variance $\sigma^{2}$ is given by 
$$ 
\P(W\le t) = \frac{1}{\sigma\sqrt{2\pi}} \int_{-\infty}^{t}e^{-\frac{x^{2}}{2\sigma^{2}}}dx . 
$$ 
A standard complex normal random variable $Z$ with mean $0$ and variance $1$ is given by 
$Z=X+iY$ where $X$ and $Y$ are two independent real normal random variables with mean 
$0$ and variance $\frac 12$.

In Theorem \ref{thm1.1} the condition \eqref{1.1} is very mild and usually we are interested in much larger sets ${\mathcal A}$.  It can also be weakened further, as in our more precise version Corollary~\ref{cor: indicator} below.  The criterion in \eqref{1.2} is more important, and may be viewed as a fourth moment condition.  The quantity in \eqref{1.2} may be thought of as the {\sl multiplicative energy} of the set ${\mathcal S}$, defined by 
\begin{equation} 
\label{1.3} 
E_{\times}(\mathcal S) = \# \{ (s_1, s_2, s_3, s_4) \in {\mathcal S}^4: \ \ s_1 s_2 =s_3 s_4 \}. 
\end{equation} 
Since there are always trivial solutions $s_1=s_3$ and $s_2=s_4$ (or $s_1= s_4$ and $s_2=s_3$), we see that $E_\times (\mathcal S) 
\ge 2 |{\mathcal S}| (|{\mathcal S}|-1) + |{\mathcal S}| = 2 |{\mathcal S}|^2 -|{\mathcal S}|$.  Thus the condition \eqref{1.2} asks for the multiplicative energy to be as small as it possibly can be.   We also point out the flexibility in choosing a subset ${\mathcal S}$ of ${\mathcal A}$ which can be quite useful because even if $|{\mathcal S}| \sim |{\mathcal A}|$ 
it can happen that $E_\times({\mathcal A})$ is much bigger than $E_\times({\mathcal S})$.


 There are several results in the literature establishing a central limit theorem for random multiplicative functions restricted to suitable subsets ${\mathcal A}$ of $[1, N]$.  
 For instance Hough \cite{Hough} considered the set ${\mathcal A}$ of integers with exactly $k$ prime factors for any fixed $k$,  and Harper\cite{Harper} extended this to allow 
 any $k=o(\log \log N)$.  Short intervals $[x,x+y]$ were considered in the work of Chatterjee and Soundararajan \cite{CS} (to be precise, they worked in the Rademacher 
 setting) who showed that a central limit theorem holds provided $y= o(x/\log x)$ (and with a technical condition that $y\ge x^{\frac 15}\log x$ should be suitably 
 large, so that $[x,x+y]$ contains many square-free integers).   Recent work of Klurman, Shkredov and Xu \cite{KSX} considers  the set of polynomial values $\{P(n):1\le n \le N\}$ 
 where $P(x)\in \Z[x]$ is a polynomial not of the form $w(x+c)^{d}$ for integers $w$, $c$ and $d$.   Using Theorem~\ref{thm1.1} we give several new examples of subsets ${\mathcal A}$ where a central limit theorem holds.

 \begin{corollary}\label{thm: 2log2-1 stein}  Let $x$ and $y$ be large, with $y\le x/(\log x)^{\alpha +\epsilon}$ where $\alpha= 2\log 2 -1$, and $\epsilon>0$ is arbitrary.   Then, as 
 $f$ varies over Steinhaus random multiplicative functions, the quantity 
$$ 
\frac{1}{\sqrt{y}} \sum_{ x\le n \le x+y} f(n), 
$$ 
is distributed like a standard complex normal random variable with mean $0$ and variance $1$.
\end{corollary}

Corollary \ref{thm: 2log2-1 stein} improves upon the result established by Chatterjee and Soundararajan \cite{CS}, where $y$ was required to be $o(x/\log x)$.  The range $y=o(x/\log x)$ is the threshold at which the fourth moment of $\sum_{x \le n\le x+y} f(n)$ matches that of a Gaussian,  and it was tentatively suggested in \cite{CS} that this might 
be the largest range in which a central limit theorem holds.   Corollary \ref{thm: 2log2-1 stein}  shows that larger values of $y$ are permissible; the 
source of the improvement is the flexibility in Theorem \ref{thm1.1} of choosing an appropriate density $1$ subset ${\mathcal S}$  of $[x,x+y]$ on which the 
fourth moment can be shown to match that of a Gaussian.  It is conceivable that the central limit theorem holds for still larger values of $y$, and perhaps for all $y \le x/(\log x)^{\epsilon}$ for any $\epsilon >0$.  On the other hand, Harper's work \cite{HarperLow} shows that $\frac{1}{\sqrt{y}} \sum_{x\le n\le x+y} f(n)$ is not Gaussian if $y \ge x/(\log \log x)^{\frac 12-\epsilon}$  (and indeed some modifications to his ideas permit even the wider range $y \ge x/\exp((\log \log x)^{\frac 12-\epsilon})$.

 Our second example concerns integers that are sums of two squares (one could consider more general sifted sequences, and we comment on this briefly in \S 5).  While we are unable to treat the set of integers in $[1,x]$ that are sums of two squares, we come close to this and can treat the sums of two squares in any short interval $[x,x+y]$ so long as $x^{\frac{1}3} < y =o(x)$.   Here our interest is in allowing $y$ as 
 large as possible, and the lower bound on $y$ is imposed in order to guarantee that $[x,x+y]$ contains the expected number of integers (namely about $y/\sqrt{\log x}$) 
 that are sums of two squares (see Hooley \cite{Hooley} who shows that the exponent $\frac 13$ may be replaced with the best available result on the circle problem). 
 
 \begin{corollary}\label{thm: sum of two squares}  Let $x$ and $y$ be large with $y$ in the range $x^{\frac 13} < y = o(x)$.  
Let $\CA$ denote the set of integers in $[x,x+y]$ that are the sum of two squares.   In the given range of $y$, we have  $|\CA|\asymp \frac{y}{\sqrt{\log x}}$, 
and moreover, as $f$ varies over Steinhaus random multiplicative functions, the quantity 
\[\frac{1}{\sqrt{|\mathcal{A}|} } \sum_{n \in \mathcal{A}} f(n) \]
is distributed like a standard complex normal random variable with mean $0$ and variance $1$. 
\end{corollary}

 If we restrict a Steinhaus random multiplicative function to its values on primes, then these values are independent random 
 variables and clearly the central limit theorem holds.  Our next example shows that the central limit theorem also holds for the 
 set of shifted primes $p+k$, for any fixed non-zero integer $k$.  
 
 \begin{corollary}\label{thm: shift primes}  Let $k$ be any fixed non-zero integer.  Let $N$ be large, and let ${\mathcal A}$ denote the set of 
 integers of the form $p+k$ in $[1, N]$.  As $f$ varies over Steinhaus random multiplicative functions,  
the quantity 
\[
\frac{1}{\sqrt{|\mathcal{A}|} } \sum_{n \in \mathcal{A}} f(n) 
\]
is distributed like a standard complex normal random variable with mean $0$ and variance $1$.
\end{corollary}

The examples in the three corollaries above give subsets of size $N/(\log N)^{\beta+\epsilon}$ in $[1,N]$ for 
which a central limit theorem holds, with the largest subset being the short interval result of Corollary \ref{thm: 2log2-1 stein} 
which permits $\beta =2\log 2 -1$.  What is the largest subset ${\mathcal A}$ of $[1,N]$ for which we can establish a central limit theorem 
for $\sum_{n\in {\mathcal A}} f(n)$?  As we indicated earlier, we expect (but cannot prove) that a central limit theorem holds for short intervals $[x,x+y]$ with $y= x/(\log x)^{\epsilon}$.  
The largest set that we have been able to find is given in the next corollary, and this set is related to the Erd{\H o}s multiplication table problem.  The construction is 
partly random, and based largely on the work of Ford \cite{Ford18}.


\begin{corollary}\label{thm:large subset}  For all large $N$, there exists a subset ${\mathcal A}$ of $[1,N]$ with
\[|
\mathcal{A}|\ge \frac{N}{(\log N)^{\theta} (\log \log N)^{7}}, \qquad \text{  where     }  \theta = 1- \frac{1+\log \log 4}{\log 4} = 0.0430\ldots
 \]
such that, for random Steinhaus multiplicative functions $f$, the quantity
\[
\frac{1}{\sqrt{|\mathcal{A}|} } \sum_{n \in \mathcal{A}} f(n)
\]
is distributed like a standard complex normal random variable with mean $0$ and variance $1$.
\end{corollary}

 Our proof of  Theorem \ref{thm1.1} applies more generally to weighted sums $\sum_{n\le N} a_n f(n)$ for suitable sequences of complex numbers $a_n$ 
 (indeed the precise version Theorem \ref{thm: a_n} is formulated for these more general sums).  We highlight a particularly interesting case when $a_n = 
 e(n\theta)$ (here $e(n\theta)$ denotes as usual $e^{2\pi i n\theta}$).  For a large class of irrational numbers $\theta$ we show that $\sum_{n\le N} f(n) e(n\theta)$ has a Gaussian distribution.

\begin{theorem}\label{thm: Diophantine}  Let $\theta$ denote an irrational number such that for some positive constant $C=C(\theta)$ and all $q \in {\mathbb N}$ we 
have 
\begin{equation} \label{1.4}
\Vert q \theta \Vert  := \min_{n\in {\mathbb Z}} |q\theta -n | \ge C \exp( - q^{\frac{1}{50}}).
\end{equation} 
If $N$ is sufficiently large (in terms of $\theta$), then as $f$ varies over Steinhaus random multiplicative functions, the quantity 
\[
\frac{1}{\sqrt{N} } \sum_{n \le N} e(n\theta) f(n) 
\]
is distributed like a standard complex normal random variable with mean $0$ and variance $1$.  
\end{theorem}

The condition \eqref{1.4} imposed on $\theta$ holds for almost all irrational numbers $\theta$, and includes all algebraic irrationals, as 
well as transcendental numbers such as $e$ and $\pi$ that are known not to be Liouville numbers.  We have made no attempt to optimize the exponent $\frac 1{50}$ appearing 
in the criterion \eqref{1.4}, and there is certainly scope for improving it.   On the other hand, from Harper's work it follows that Theorem \ref{thm: Diophantine} 
cannot hold for rational $\theta$, as well as $\theta$ that permit extremely good rational approximations, so that some version of criterion \eqref{1.4} is necessary.  
Finally, we point out recent related work of Benatar, Nishry, and Rodgers \cite{BNR} who consider, for a random Steinhaus multiplicative function $f$, distribution questions 
concerning $\sum_{n\le N} f(n) e(n\theta)$ as $\theta$ varies in ${\mathbb R}/{\mathbb Z}$.

 Briefly, the paper is organized as follows.  Section 2 develops the martingale central limit theorem in a quantitative form, based on the work of McLeish  \cite{McLeish}. 
 Section 3 makes an initial application of these results to the setting of random multiplicative functions.  In particular, we derive there our main technical result Theorem \ref{thm: a_n}, and deduce the simplified Theorem~\ref{thm1.1} stated above.  Section 4 is devoted to the distribution of random multiplicative functions in short intervals, and we establish Corollary \ref{thm: 2log2-1 stein} there.  Corollary \ref{thm: sum of two squares} is treated in Section 5.  In both Sections 4 and 5 a crucial role is played by the flexibility of being able to choose 
 dense subsets ${\mathcal S}$ of ${\mathcal A}$ in Theorem  \ref{thm: a_n}. 
 Corollary \ref{thm: shift primes} admits a short proof based on a simple upper bound sieve, which is 
 presented in Section 6.  Section 7 gives the proof of Corollary \ref{thm:large subset}, and the construction is based on the work of Ford \cite{Ford18} on extremal product sets. 
 Section 8 deals with Theorem~\ref{thm: Diophantine}, and the work of Montgomery and Vaughan \cite{MV77} on exponential sums with multiplicative functions plays a key role here.  
Finally, Section 9 ends with a brief discussion of corresponding results in the setting of Rademacher random multiplicative functions.

\subsection*{Acknowledgments}
We  thank Adam Harper for helpful discussions and comments on an earlier version of the paper. We are also grateful to Louis Gaudet 
for raising a  question during the second author's graduate student seminar at AIM, which led us to Corollary~\ref{thm: sum of two squares}.  Thanks are also due to the referee for a careful reading.  
K.S.  is partially supported through a grant from the National Science Foundation, and a Simons Investigator Grant from the Simons Foundation. 
M.W.X. is partially supported by the Cuthbert C. Hurd Graduate Fellowship in the Mathematical Sciences, Stanford.

\section{McLeish's martingale central limit theorem}\label{sec: McLeish}	

In this section we give a quantitative version of the martingale central limit theorem, which we will apply to the study of random multiplicative functions.  
We follow the treatment in McLeish \cite{McLeish}, but adding some quantification.  The short proof is included for completeness, and in the hope that it may be useful to readers more familiar with analytic number theory than probability.

Let $X_1$, $\ldots$, $X_N$ denote a martingale difference sequence of real valued random variables.  That is, we suppose that 
$$ 
\E[X_1 ]=0, 
$$
and for $1\le n\le N-1$, 
$$ 
\E[ X_{n+1} | X_1, \ldots, X_n] = 0, 
$$ 
where $\E[X|Y]$ denotes the conditional expectation of $X$ given $Y$.
Define 
$$
S_N = X_1 +\ldots + X_N, 
$$ 
and our goal is to show that, under suitable conditions, $S_N$ behaves like a real Gaussian with mean $0$ and variance $1$.  We will achieve this by computing the Fourier transform $\E [ e^{it S_N}]$ and showing that it approximates $e^{-t^2/2}$ (which is the Fourier transform of a standard real Gaussian).   We begin with a simple lemma, which will be key to the result.

\begin{lemma}  \label{lem2.1} Let $y_1$, $\ldots$, $y_N$ be real numbers, and define $K$ (in $[1,N]$) to be the largest integer such that 
$$ 
\sum_{n=1}^{K-1} y_n^2 < 2.
$$ 
Then for any real number $t$ we have 
$$ 
e^{it (y_1 +\ldots + y_N)} = \prod_{n=1}^{K} (1+ ity_n) e^{-\frac {t^2}{2}} + O\Big( e^{t^2} \max_{n=1}^{N} |y_n| 
\Big) + O\Big( e^{t^2} \min \Big( 1, \Big|\sum_{n=1}^{N} y_n^2 -1 \Big|\Big)\Big). 
$$
\end{lemma} 
\begin{proof}  Throughout the proof, the following elementary observations will be useful.  For any real 
number $x$ we have 
\begin{equation} 
\label{2.1} 
|1+ ix| = (1+x^2)^{\frac 12} \le e^{x^2/2}, 
\end{equation} 
and, by a Taylor expansion,  
\begin{equation} 
\label{2.2} 
e^{ix} = (1+ix) e^{-x^2/2} \exp(O(|x|^3)). 
\end{equation}

First let us consider the case $\sum_{n=1}^{N} y_n^2 \ge 2$, where the remainder terms in the lemma are 
clearly $\gg e^{t^2} (1 + \max_{n=1}^{N} |y_n|)$.  On the other hand, note that by \eqref{2.1} and the definition of $K$
\begin{align*}
\Big| e^{it(y_1+\ldots +y_N)} - \prod_{n=1}^{K} (1+ ity_n) e^{-\frac {t^2}{2}} \Big| 
&\le 1 + (1+ |ty_K|) e^{-\frac {t^2}{2}}\prod_{n=1}^{K-1} |1+ity_n| 
\\
&\le 1+ (1+ |ty_K|) e^{-\frac {t^2}{2}} \exp\Big(\frac 12 \sum_{n=1}^{K-1} t^2 y_n^2\Big) 
\\
&\le 1 + (1+|ty_K|) e^{t^2/2}. 
\end{align*} 
Since $e^{t^2} \gg 1 + te^{t^2/2}$, the result follows in this case. 

Consider then the complementary case $\sum_{n=1}^{N} y_n^2 <2$, so that $K=N$.   Here \eqref{2.1} gives 
$$
\Big| \prod_{n=1}^{N} (1+i t y_n)\Big|  \le \exp\Big( \sum_{n=1}^{N} \frac{t^2 y_n^2}{2} \Big) \le e^{t^2}, 
$$  
so that the result holds if $\max_{n=1}^{N} |y_n| \ge e^{-\frac{t^2}{2}}$, or if 
$\sum_{n=1}^{N} y_n^2 \ge \frac 32$.  Assume therefore that $\max_{n=1}^{N} |y_n| 
\le e^{-\frac {t^2}{2}}$, and that $\sum_{n=1}^{N} y_n^2 \le \frac 32$.   

Now the Taylor approximation \eqref{2.2} gives 
$$ 
e^{it(y_1+ \ldots + y_N)} = \prod_{n=1}^{N} (1+i ty_n) \exp\Big( -\frac 12 \sum_{n=1}^{N} t^2y_n^2 + O\Big(\sum_{n=1}^{N} |ty_n|^3\Big) \Big).
$$
Since 
$$ 
|t|^3 \sum_{n=1}^{N} |y_n|^3 \le  |t|^3 \Big( \max_n |y_n|\Big) \sum_{n=1}^{N} y_n^2 \le 2 |t|^3 \max_{n} |y_n|, 
$$ 
and this is $\ll 1$ by our assumption, we conclude that 
\begin{align*}
e^{it(y_1+ \ldots + y_N)} &= \prod_{n=1}^{N} (1+i ty_n) \exp\Big( -\frac 12 \sum_{n=1}^{N} t^2y_n^2\Big) \Big( 1+ O\Big( |t|^3 \max_n |y_n|\Big)\Big)\\
&= \prod_{n=1}^{N} (1+i ty_n) \exp\Big( -\frac 12 \sum_{n=1}^{N} t^2y_n^2\Big) + O\Big( |t|^3 \max_n |y_n|\Big). 
\end{align*}

Since 
$$ 
\Big| \prod_{n=1}^{N} (1+ ity_n) \Big| \le \exp\Big( \frac{t^2}{2} \sum_{n=1}^{N} y_n^2 \Big) \le \exp\Big( \frac 34 t^2\Big), 
$$ 
and 
$$ 
\Big| \exp\Big( -\frac{t^2}{2} \sum_{n=1}^{N} y_n^2 \Big) - \exp\Big(-\frac{t^2}{2} \Big) \Big| \le \frac{t^2}{2} \Big| \sum_{n=1}^{N} y_n^2 -1\Big|, 
$$ 
the desired estimate follows. 
\end{proof}

\begin{proposition} \label{prop2.2} Let $X_n$ be a real valued martingale difference sequence, and put $S_N = X_1 +\ldots +X_N$.  Assume that 
$\E [ \max_{n=1}^{N} |X_n|]$ exists.  Then for any $t \in {\mathbb R}$ we have 
$$ 
\E[ e^{it S_N} ] = e^{-t^2/2} + O\Big( e^{t^2} \Big( \E [ \max_{n=1}^{N} |X_n|] + \E \Big[ \min \Big(1, \Big| \sum_{n=1}^{N} X_n^2 -1\Big|\Big) \Big]\Big)\Big). 
$$ 
\end{proposition} 
\begin{proof}  Define a new sequence of random variables $\widetilde{X_n}$ by 
$$ 
\widetilde{X_n} = X_n {\mathbb I} \Big[ \sum_{j=1}^{n-1} X_j^2 \le 2 \Big ]. 
$$
From its definition one sees that $\widetilde{X_n}$ is also a martingale difference sequence.  

Further, if $K$ is the largest integer in $[1,N]$ with $\sum_{n=1}^{K-1} X_n^2 \le 2$, then note that ${\widetilde X}_n 
= X_n$ for $n\le K$, and ${\widetilde X_n} =0$ for $n >K$.  From Lemma \ref{lem2.1} it follows that 
$$ 
e^{it S_N} = \prod_{n=1}^{N} (1+ it {\widetilde X}_n) e^{-\frac {t^2}{2} } + O\Big( e^{t^2} 
\Big( \max_n |X_n| + \min\Big( 1, \Big| \sum_{n=1}^{N} X_n^2 -1\Big| \Big) \Big)\Big). 
$$
Take expectations on both sides, and note that the martingale property gives 
$$ 
\E \Big[ \prod_{n=1}^{N} (1+ it \widetilde{X_n})\Big] = 1. 
$$ 
The proposition follows.
\end{proof} 	
	
From Proposition \ref{prop2.2} 	we extract the following quantitative result where the conditions are stronger than necessary but easier 
to check in practice. 

\begin{theorem} \label{thm2.3} Let $X_n$ be a real valued martingale difference sequence, and put $S_N = X_1 +\ldots +X_N$.
Suppose that $\E[ X_n^4]$ exists for each $n$.  Then for any real number $t$ we have 
$$ 
\E[e^{itS_N}] = e^{-t^2/2} + O\Big( e^{t^2} \Big( \sum_{n=1}^{N} \E [ X_n^4] \Big)^{\frac 14} \Big) + O \Big( e^{t^2} \Big( 
\E\Big[ \Big(\sum_{n=1}^N X_n^2 -1\Big)^2 \Big]\Big)^{\frac 12} \Big). 
$$ 
\end{theorem} 
\begin{proof}   H{\" o}lder's inequality gives 
$$ 
\E [ \max_{n=1}^{N} |X_n| ] \le \E \Big[ \Big( \sum_{n=1}^{N} X_n^4 \Big)^\frac 14 \Big] \le \Big( \E \Big[ \sum_{n=1}^{N} X_n^4\Big]\Big)^{\frac 14} 
= \Big( \sum_{n=1}^{N} \E [ X_n^4] \Big)^{\frac 14}. 
$$ 
Further, applying the Cauchy-Schwarz inequality we find  
$$ 
\E\Big[ \min \Big( 1, \Big|\sum_{n=1}^{N} X_n^2 -1\Big|\Big) \Big] \le 
\Big( \E \Big[ \Big(\sum_{n=1}^{N} X_n^2- 1\Big)^2 \Big] \Big)^{\frac 12}. 
$$ 
Thus the stated result follows immediately from Proposition \ref{prop2.2}. 
\end{proof} 

So far we have treated martingale difference sequences of real valued random variables.  Let us now turn to a martingale difference sequence $Z_1$, $\ldots$, $Z_N$ 
of complex valued random variables, where we aim to show that the partial sums $S_N = \sum_{n=1}^N Z_n$ have a standard complex normal distribution.   To achieve this we 
shall compute the Fourier transform $\E[ e^{it_1 \text{Re}(S_N) + it_2 \text{Im}(S_N)}]$ and show that this approximates $e^{-(t_1^2 +t_2^2)/4}$ (which is the Fourier transform of a standard complex Gaussian).  

\begin{theorem} 
\label{thm2.4} Let $Z_1$, $\ldots$, $Z_n$ be a martingale difference sequence of complex valued random variables, and put $S_N = \sum_{n=1}^{N} Z_n$.  
Assume that $\E[ |Z_n|^4]$ exists for each $n$.  Then for any real numbers $t_1$ and $t_2$ we have, with $t^2 = (t_1^2+t_2^2)/2$, 
\begin{align*}
\E[ e^{it_1 \text{Re}(S_N)+it_2 \text{Im}(S_N)}] &= e^{-t^2/2} + O\Big( e^{t^2} \Big( \sum_{n=1}^{N}\E[ |Z_n|^4] \Big)^{\frac 14} \Big) 
+ O\Big( e^{t^2} \Big( \E \Big[ \Big( \sum_{n=1}^N |Z_n|^2 - 1 \Big)^2 \Big]\Big)^{\frac 12}\Big)
\\
& + O\Big( e^{t^2} \max_\phi \Big(\E \Big[ \Big( \sum_{n=1}^{N} (e^{-i\phi} Z_n^2 + e^{i\phi} \overline{Z_n}^2) \Big)^2 \Big] \Big)^{\frac 12}\Big).\\
\end{align*}
\end{theorem} 
\begin{proof} Write 
$$ 
\frac{t_1 + it_2}{2} = \frac{t e^{i\theta}}{\sqrt{2}} \qquad \text{ with } t = \frac{\sqrt{t_1^2+t_2^2}}{\sqrt{2}},
$$ 
so that 
$$ 
t_1 \text{Re}(S_N) + t_2 \text{Im}(S_N) = t \Big( \frac{e^{-i\theta}S_N + e^{i\theta}\overline{S_N}}{\sqrt{2}}\Big) = t \sum_{n=1}^{N} \frac{e^{-i\theta } Z_n + e^{i\theta} \overline{Z_n}}{\sqrt{2}}. 
$$ 
Now $X_n = (e^{-i\theta } Z_n + e^{i\theta} \overline{Z_n})/\sqrt{2}$ forms a real valued martingale difference sequence, and we may apply Theorem \ref{thm2.3}.  It follows that 
$$ 
\E[ e^{i(t_1 \text{Re}(S_N) + t_2 \text{Im}(S_N))}] = e^{-t^2/2} + O\Big( e^{t^2} \Big( \sum_{n=1}^{N} \E[ X_n^4] \Big)^{\frac 14}\Big) 
+ O\Big( e^{t^2}\Big( \E \Big[ \Big( \sum_{n=1}^{N} X_n^2 - 1 \Big)^2 \Big] \Big)^{\frac 12} \Big).
$$ 
Now note that $\E[X_n^4] \ll \E[ |Z_n|^4]$ and so the first error term above is 
$$ 
O\Big( e^{t^2} \Big(  \sum_{n=1}^{N} \E[ |Z_n|^4] \Big)^{\frac 14}\Big).
$$ 
Regarding the second error term, note that 
\begin{align*}
 \Big( \sum_{n=1}^N X_n^2 -1 \Big)^2  &= \Big( \sum_{n=1}^{N} \Big( \frac{e^{-2i \theta} Z_n^2 + e^{2i\theta} \overline{Z_n}^2}{2} + |Z_n|^2 \Big) -1 \Big)^2 
\\
& \ll \Big( \sum_{n=1}^{N} |Z_n|^2 - 1 \Big)^2 + \Big( \sum_{n=1}^{N} (e^{-2i\theta} Z_n^2 +e^{2i\theta} \overline{Z_n}^2) \Big)^2. 
 \end{align*}
The theorem follows readily. 
\end{proof}

\section{Application to random multiplicative functions} \label{section3}

We now apply the work in Section~\ref{sec: McLeish} to the study of random multiplicative functions.  From now on, we shall denote by $P(n)$ the largest prime factor of 
the integer $n$. 

\begin{theorem} \label{thm: a_n} Let $f$ denote a random Steinhaus multiplicative function, and let $a_n$ denote a sequence of complex numbers.  
Put 
$$ 
V= \sum_{n\le N} |a_n|^2,
$$ 
and define the complex valued random variable 
$$ 
Z:= \frac{1}{\sqrt{V}} \sum_{n\le N} a_n f(n). 
$$ 
Suppose ${\mathcal S}$ is a subset of $[2,N]$ such that for some $1 \ge \epsilon >0$ the following three conditions hold:

(1).  We have 
$$ 
\sum_{\substack{ n\le N \\ n\notin {\mathcal S}}} |a_n|^2 \le \epsilon^2 V. 
$$ 

(2).  We have 
$$ 
\Big| \sum_{\substack{ m_1, m_2, n_1, n_2 \in {\mathcal S} \\ m_1 m_2= n_1 n_2 \\ m_1 \neq n_1, m_2\neq n_2 \\ P(m_1)=P(n_1) \\ P(m_2)=P(n_2)}} a_{m_1} a_{m_2} 
\overline{a_{n_1}a_{n_2}} \Big| \le \epsilon^2 V^2.
$$ 

(3).  We have 
$$ 
\Big| \sum_{\substack{ m_1, m_2, n_1, n_2 \in {\mathcal S} \\ m_1 m_2= n_1 n_2  \\ P(m_1)=P(n_1) =P(m_2)=P(n_2)}} a_{m_1} a_{m_2} 
\overline{a_{n_1}a_{n_2}} \Big| \le \epsilon^4 V^2.
$$ 

Then for any real numbers $t_1$ and $t_2$ we have, with $t^2 =(t_1^2+ t_2^2)/2$ 
$$ 
\E[ e^{it_1 \text{Re}(Z) + it_2 \text{Im}(Z)} ]= e^{-t^2/2} + O(e^{t^2} \epsilon). 
$$ 
\end{theorem} 
\begin{proof} Put 
$$ 
\widetilde{Z} = \frac{1}{\sqrt{V}} \sum_{n\in {\mathcal S}} a_n f(n). 
$$ 
Note that, using the Cauchy--Schwarz inequality and assumption (1),  
\begin{align*}
\Big| \E [ e^{it_1 \text{Re}(Z) + it_2 \text{Im}(Z)} ] - \E[ e^{it_1 \text{Re}(\widetilde{Z}) + it_2 \text{Im}({\widetilde Z})} ] \Big| &\ll 
\E[ |t_1 \text{Re} (Z-\widetilde{Z}) + t_2 \text{Im}(Z- \widetilde{Z})|]\\
& \ll \Big( t^2  \E[ |Z-\widetilde{Z}|^2] \Big)^{\frac 12} = \Big( \frac{t^2}{V} \sum_{\substack{ n\le N \\ n\notin {\mathcal S}}} |a_n|^2 \Big)^{\frac 12}= 
 O(\epsilon e^{t^2}). 
\end{align*} 
Thus it is enough to compute $\E[ e^{it_1 \text{Re}(\widetilde{Z}) + it_2 \text{Im}({\widetilde Z})} ]$, and we approach this using our work in Section~\ref{sec: McLeish}. 

For each prime $p\le N$ define 
$$ 
{\widetilde Z}_p = \frac{1}{\sqrt{V}} \sum_{\substack{ n \in {\mathcal S} \\ P(n) = p}} a_n f(n), 
$$ 
so that ${\widetilde Z} = \sum_{p\le N} {\widetilde Z}_p$.  Notice that each term in the sum defining ${\widetilde Z}_p$ involves $f(p)$, which is independent of all $f(\ell)$ 
with $\ell$ being a prime $<p$.  Thus ${\widetilde Z}_p$ forms a martingale difference sequence as $p$ varies over all the primes at most $N$.  Therefore, we may 
apply Theorem \ref{thm2.4} to evaluate $\E[ e^{it_1 \text{Re}(\widetilde{Z}) + it_2 \text{Im}({\widetilde Z})}]$.  The martingale decomposition given above was pioneered by Harper 
\cite{Harper}, motivated by work of Blei and Janson \cite{BJ2004}, and related decompositions have appeared for instance in \cite{LTW2013}.

Now observe that 
\begin{align*}
\sum_{p\le N} \E [ |{\widetilde Z}_p|^4] &= \frac{1}{V^2} \sum_{p\le N}  \sum_{\substack{m_1, m_2, n_1, n_2 \in {\mathcal S} \\ P(m_1)=P(m_2)=P(n_1)=P(n_2)=p}} 
a_{m_1} a_{m_2} \overline{a_{n_1} a_{n_2}} \E [ f(m_1m_2) \overline{f(n_1n_2)}] \\
&= \frac{1}{V^2}  \sum_{\substack{m_1, m_2, n_1, n_2 \in {\mathcal S} \\ P(m_1)=P(m_2)=P(n_1)=P(n_2) \\ m_1 m_2 =n_1n_2 }} a_{m_1} a_{m_2} \overline{a_{n_1} a_{n_2}},
\end{align*}
which is bounded (in magnitude) by $\epsilon^4$ by our assumption (3).  Thus the first error term in applying Theorem \ref{thm2.4} is $O(e^{t^2} \epsilon)$. 

Next observe that 
$$ 
\E \Big[ \Big( \sum_{p\le N} |{\widetilde Z}_p|^2 - 1 \Big)^2 \Big]  = \sum_{p,q \le N} \E[ |{\widetilde Z}_p|^2 |{\widetilde Z}_q|^2] - 2 \sum_{p\le N} \E[|{\widetilde Z}_p|^2] + 1. 
$$
Now 
\begin{align*}
\sum_{p,q \le N} \E[ |{\widetilde Z}_p|^2 |{\widetilde Z}_q|^2]  &= \frac{1}{V^2} 
\sum_{p,  q\le N} \sum_{\substack{m_1, n_1 \in {\mathcal S} \\ P(n_1)=P(m_1)= p}} \sum_{\substack{m_2, n_2 \in {\mathcal S}\\ P(m_2) = P(n_2) = q}} 
a_{m_1} a_{m_2} \overline{a_{n_1} a_{n_2}}  \E [ f(m_1 m_2) \overline{f(n_1 n_2)}] \\
&= \frac{1}{V^2} \sum_{\substack{ m_1, n_1, m_2, n_2 \in {\mathcal S} \\ m_1 m_2 =n_1n_2 \\ P(m_1)= P(n_1) \\ P(m_2) =P(n_2)}} a_{m_1} a_{m_2} \overline{a_{n_1} a_{n_2}} = \frac{1}{V^2} \Big( \sum_{n\in {\mathcal S}} |a_n|^2\Big)^2 + O(\epsilon^2),
\end{align*} 
upon isolating the terms $m_1 = n_1$ and $m_2 =n_2$, and then using assumption (2) to bound the remaining terms.  Further 
$$ 
\sum_{p\le N} \E[|{\widetilde Z}_p|^2] = \frac{1}{V} \sum_{p\le N} \sum_{\substack{m, n \in {\mathcal S} \\ P(m) =P(n)=p}} a_m \overline{a_n} 
\E[ f(m) \overline{f(n)}] = \frac 1V \sum_{n\in {\mathcal S}} |a_n|^2,
$$ 
so that 
$$ 
\E \Big[ \Big( \sum_{p\le N} |{\widetilde Z}_p|^2 - 1 \Big)^2 \Big] = \Big( \frac 1V \sum_{n\in {\mathcal S}} |a_n|^2 -1 \Big)^2 + O(\epsilon^2) = O(\epsilon^2), 
$$ 
upon using assumption (1) in the last step.  Therefore the second error term while using Theorem \ref{thm2.4} may also be bounded by $O(e^{t^2} \epsilon)$.

Finally consider the third error term in Theorem \ref{thm2.4}, which involves the maximum over $\phi$ of 
$$ 
\E \Big[ \Big( \sum_{p\le N} (e^{-i\phi} {\widetilde Z}_p^2 + e^{i\phi} {\overline {\widetilde Z}_p}^2)\Big)^2 \Big] 
= \sum_{p, q \le N} \E \Big[ (e^{-i\phi} {\widetilde Z}_p^2 + e^{i\phi} {\overline {\widetilde Z}_p}^2) (e^{-i\phi} {\widetilde Z}_q^2 + e^{i\phi} {\overline {\widetilde Z}_q}^2)\Big].
$$ 
If $p\neq q$ then 
$$ 
\E[{\widetilde Z}_p^2 {\widetilde Z}_q^2] = \E [ {\widetilde Z}_p^2 \overline{\widetilde {Z}_q}^2 ] 
= \E[ \overline{\widetilde Z}_p^2 \widetilde{Z}_q^2 ] = \E[ \overline{\widetilde Z}_p^2  \overline{\widetilde {Z}_q}^2] =0, 
$$ 
as may be seen by expanding these terms and noting that no diagonal terms arise.  We are left with the terms $p=q$ 
which contribute 
$$ 
\ll \sum_{p\le N}\E \Big[ |{\widetilde Z}_p|^4 \Big] \ll \epsilon^4,
$$ 
from our work on the first error term.   Thus the third error term appearing in Theorem \ref{thm2.4} is $O(e^{t^2} \epsilon^2)$, and the proof of the theorem is complete.
\end{proof} 

Let us now record a simplified version of Theorem \ref{thm: a_n} when $a_n$ is the indicator function of a set. 

\begin{corollary} \label{cor: indicator}  Let ${\mathcal A}$ be a non-empty subset of natural numbers and let $f$ denote a random Steinhaus multiplicative function.  Define the complex valued random variable 
$$ 
Z = \frac{1}{\sqrt{|{\mathcal A}|}} \sum_{n\in {\mathcal A}} f(n). 
$$ 
Let ${\mathcal S}$ be a subset of ${\mathcal A}$, with all elements in ${\mathcal S}$ being at least $2$.  Suppose that $1\ge \epsilon \ge 0$ is such that the following three conditions are met: 

(1).  $|{\mathcal A} \backslash {\mathcal S}| \le \epsilon^2 |{\mathcal A}|$. 

(2). The number of solutions to $m_1 m_2 = n_1 n_2$ with $m_1$, $m_2$, $n_1$ and $n_2 \in {\mathcal S}$, and $m_1 \neq n_1$, $m_1 \neq n_2$ is bounded by 
$\epsilon^4 |{\mathcal A}|^2$.  

(3).  For each prime $p$
$$ 
\# \{ s\in {\mathcal S}: \ \ P(s) = p\} \le \epsilon^4 |{\mathcal A}|. 
$$ 
Then for any real numbers $t_1$ and $t_2$ we have, with $t^2 =(t_1^2+ t_2^2)/2$ 
$$ 
\E[ e^{it_1 \text{Re}(Z) + it_2 \text{Im}(Z)} ]= e^{-t^2/2} + O(e^{t^2} \epsilon). 
$$ 
\end{corollary} 
\begin{proof}  We apply Theorem \ref{thm: a_n} with $a_n$ denoting the indicator function of the set ${\mathcal A}$.  Condition (1) of Theorem \ref{thm: a_n} holds 
by assumption (1) here.  

The sum in condition (2) of Theorem \ref{thm: a_n} counts non-diagonal solutions of $m_1m_2 = n_1n_2$ with $m_i, n_i \in {\mathcal S}$ together with special diagonal solutions 
$m_1=n_2$ and $m_2=n_1$ with $P(m_1)=P(m_2)=P(n_1)=P(n_2)$.  By our assumption (2) the non-diagonal solutions are bounded by $\epsilon^4 |{\mathcal A}|^2$.  
As for the special diagonal solutions, these are bounded by
 $$ 
 \sum_p (\#\{ s\in {\mathcal S}: P(s)=p\})^2 \le \epsilon^4 |{\mathcal A}| \sum_p \# \{ s\in {\mathcal S}: P(s)=p\} \le \epsilon^4 |{\mathcal A}|^2,
 $$ 
 upon using our assumption (3).  Thus condition (2) of Theorem \ref{thm: a_n} holds, with $2\epsilon^4$ in place of $\epsilon^2$ there. 
 
 Finally the sum in condition (3) of Theorem \ref{thm: a_n} is bounded above by the count of non-diagonal solutions to $m_1 m_2 =n_1 n_2$, together with two copies of the special diagonal solutions bounded above.  Thus this sum is bounded by $3 \epsilon^4 |{\mathcal A}|^2$.  
 
We have checked that the conditions in Theorem \ref{thm: a_n} are met (with a slightly larger value of $\epsilon$ there), and the stated result follows. 
\end{proof}

The key condition in Corollary \ref{cor: indicator} is the assumption (2) on non-diagonal solutions.  The third assumption is often harmless, and our next lemma shows that it holds automatically for large subsets of intervals.  

\begin{lemma} \label{lem3.1: large prime}  For all primes $p$, and all $x\ge y\ge 3$ 
$$ 
\# \{ x< n \le x+y: \ \ P(n) = p\} \ll y \exp(-\tfrac 12 \sqrt{\log y \log \log y}). 
$$ 
\end{lemma} 
\begin{proof}  If $p> \exp(\tfrac 12 \sqrt{\log y \log \log y})$ then 
$$ 
\# \{ x< n \le x+y: \ \ P(n) = p\} \le \frac{y}{p} +1 \ll y \exp(-\tfrac 12 \sqrt{\log y \log \log y}).
$$ 
If $p< \exp(\tfrac 12 \sqrt{\log y \log \log y})=: z$ (say) then 
$$ 
\# \{ x< n \le x+y: \ \ P(n) = p\} \le \Psi(x+y,z) - \Psi(x,z) \le \Psi(y,z),
$$
where $\Psi(x,z)$ denotes the number of integers below $x$ all of whose prime factors are below $z$, and the inequality used 
above is due to Hildebrand  \cite{Hildebrand85}.  The bound of the lemma now follows from the familiar estimate 
$$ 
\Psi(y,z) = y u^{-(1+o(1))u} 
$$ 
with $u= \log y/\log z = 2\sqrt{\log y/\log \log y}$. 
\end{proof} 

\begin{proof}[Proof of Theorem~\ref{thm1.1}] Since the convergence of characteristic functions implies the convergence in distribution (see, for example, \cite{Gut}),
the theorem follows immediately from Corollary~\ref{cor: indicator} and Lemma~\ref{lem3.1: large prime}.
\end{proof}

Several times above we have encountered the relation $m_1m_2 = n_1n_2$, and we now record a simple parametrization 
of these solutions, setting up notation that we shall use later. 

\begin{lemma} \label{lem3.2: parameter}  The solutions to $m_1 m_2 = n_1 n_2$ may be parameterized as 
$$ 
m_1 = g a, \ \ \  m_2 = h b, \ \ \ n_1 = g b, \ \ \ n_2 = ha, 
$$ 
where $(a,b)=1$.  Diagonal solutions correspond to solutions with $g=h$ or $a=b$ (in which case $a=b=1$). 
\end{lemma} 
\begin{proof}  All we have done is to write $g=(m_1,n_1)$ and $h=(m_2, n_2)$. 
\end{proof} 
	
\section{Random multiplicative functions in short intervals:  Proof of Corollary~\ref{thm: 2log2-1 stein}}\label{sec: martingale}

We now use our work in Section \ref{section3} to study partial sums of random Steinhaus multiplicative functions over short intervals ${\mathcal A} =[x, x+y]$.  
Throughout we think of $y$ as large, and $x \ge y$.  Our goal is to show Corollary~\ref{thm: 2log2-1 stein}, which states that in the range $y\le x/(\log x)^{2\log 2-1+\epsilon}$ the limiting distribution of $\sum_{x\le n\le x+y} f(n)$ is Gaussian (improving upon the earlier result in \cite{CS}).  Below we use the standard notation $\Omega(n)$ to denote the 
number of prime factors of $n$ counted with multiplicity.

\begin{proposition}\label{prop:typical count}  Let $y$ be large with $y\le x$.   If $y\le x/(\log x)^2$ take ${\mathcal S}$ to be all of ${\mathcal A}= [x,x+y]$, and in the 
range $y > x/(\log x)^2$ define ${\mathcal S}$ to be subset of elements in ${\mathcal A}$ satisfying $\Omega(n) \le (1+\epsilon) \log \log x$.   
Then, the number of integers in $[x,x+y]$ that are not in ${\mathcal S}$ is $o(y)$.  
Further, in the range $y\le x(\log x)^{1-2\log 2 -\epsilon}$, the number of off-diagonal solutions to 
$$
 m_1m_2= n_1 n_2  \qquad  \text{   with  } m_1, m_2, n_1, n_2 \in {\mathcal S}, 
$$ 
is $o(y^2)$. 
\end{proposition}

\begin{proof}[Deduction of Corollary~\ref{thm: 2log2-1 stein}]
We apply Corollary~\ref{cor: indicator} with $\CA=[x,x+y]$ and $\CS$ as defined above. The first and the second conditions in Corollary~\ref{cor: indicator} follow from Proposition~\ref{prop:typical count}, and the third condition follows from Lemma~\ref{lem3.1: large prime}.  Thus Corollary~\ref{cor: indicator} applies and shows that 
the Fourier transform of $\sum_{x \le  n\le x+y} f(n)$ matches that of a complex Gaussian, giving the desired result. 
\end{proof}

The proof of Proposition~\ref{prop:typical count} relies on the following result of Shiu \cite{Shiu1980}.
	\begin{lemma}[Shiu]\label{Shiu}
Let $f(n)$ be a non-negative multiplicative function such that 

(i) $f(p^{\ell}) \le A_1^{\ell}$ for some positive constant $A_1$, and 

(ii) for any $\epsilon>0$, $f(n)\le A_2 n^{\epsilon}$ for some $A_2=A_2(\epsilon)$. 

Then, for all $\sqrt{x} \le y\le x$ we have 
 \[
\sum_{\substack{x\le n\le  x+y}} f(n) \ll \frac{y}{ \log x} \exp\Big (\sum_{p\le x} \frac{f(p)}{p} \Big). 
 \]
\end{lemma}
	
	\begin{proof}[Proof of Proposition~\ref{prop:typical count}]  We first show that $|{\mathcal A} \backslash {\mathcal S}| = o(y)$.  In the range $y\le x/(\log x)^2$ we have 
	${\mathcal S}= {\mathcal A}$ and so this holds trivially.  Suppose then that $x/(\log x)^2 \le y\le x$, where the claim is simply that most integers in $[x,x+y]$ have 
	the expected number of prime factors, namely $\sim \log \log x$.   We may deduce this quickly from Lemma \ref{Shiu}, taking there $f$ to be 
	the completely multiplicative function $f(n) =\exp(\epsilon^{\prime} \Omega(n))$ with $\epsilon^{\prime} =1/\log \log \log x$.    Then, an application of Lemma \ref{Shiu} gives  
	\begin{align*}
	|{\mathcal A}\backslash{\mathcal S}| &\le \exp(-\epsilon^{\prime} (1+\epsilon)\log \log x) \sum_{x< n\le x+y} f(n) 
	\\
	&\ll \exp(-\epsilon^{\prime} (1+\epsilon)\log \log x) \frac{y}{\log x} 
	\exp\Big( \sum_{p\le x} \frac{f(p)}{p}\Big) \ll y \exp\Big(-\frac{ \epsilon^{\prime} \epsilon}{2}  \log \log x \Big), 
	\end{align*} 
	proving our claim (with lots of room to spare).  

We now focus on the main thrust of the proposition, which is to estimate the number of non-diagonal solutions to $m_1 m_2= n_1n_2$ with all 
variables being in ${\mathcal S}$.  We use the parameterization in Lemma~\ref{lem3.2: parameter}, and write below $\delta =y/x$.  Thus our goal is to bound 
$$ 
\sum_{\substack{a \neq b \\ (a,b)=1 }} \ \ \sum_{\substack{ g \neq h \\ ga, gb, ha, hb \in {\mathcal S}}} 1. 
$$ 
Since $m_1$, $m_2$, $n_1$, and $n_2$ must all lie in the interval $[x,x+y]$, we must have 
\begin{equation} \label{3.1} 
\frac{g}{h} = \frac{m_1}{n_2} \in [1/(1+\delta), (1+\delta)], \qquad \text{  and  } \qquad \frac ab = \frac{m_1}{n_1} \in [1/(1+\delta),(1+\delta)]. 
\end{equation} 
Since we are only interested in off-diagonal solutions (with $a\neq b$ and $g\neq h$), we must have $a$, $b$, $g$, and $h$ all being $\gg 1/\delta$.  In particular, there are no off-diagonal solutions if $y\le c\sqrt{x}$ for a suitable constant $c$, since then $ga$ (for instance) would be $\gg 1/\delta^2 > x$.  Assume henceforth that $y \gg \sqrt{x}$.  Ignoring the condition $(a,b)=1$, our goal now is to bound 
\begin{equation} \label{3.2} 
\sum_{\substack{ g, a, b, h \gg 1/\delta \\  ga, gb, ha, hb \in {\mathcal S}}} 1 \ll \sum_{\substack{ g, h \\ 1/\delta \ll g, h \ll \sqrt{x}}}\ \  \sum_{\substack{a, b \\ ga, gb, ha, hb \in {\mathcal S}} } 1. 
\end{equation} 
In the last estimate above, we used that $g/h$ and $a/b$ are both in $[(1+\delta)^{-1}, (1+\delta)]$, and since $ga \le x$ we must have either $a$ and $b$ being $\ll \sqrt{x}$ 
or $g$ and $h$ being $\ll \sqrt{x}$; by symmetry we restricted attention to the latter case.

If $g$ is given, then by \eqref{3.1} there are $\ll \delta g$ choices for $h$.  If $g$ and $h$ are fixed, then there are $\ll y/g$ choices each for $a$ and $b$.  Therefore, the number of solutions in \eqref{3.2} may be bounded by 
\begin{equation} 
\label{4.3}
\ll \sum_{1/\delta \ll g\ll \sqrt{x}} (\delta g ) (y/g)^2 \ll \delta y^2 (\log x). 
\end{equation} 
This is $o(y^2)$ when $y = o(x/\log x)$, and establishes the proposition in that range.

Now suppose that $x/(\log x)^2 \le y\le x$.  In this range, we exploit that 
${\mathcal S}$ contains only those integers $n\in [x,x+y]$ with $\Omega(n)\le K:= (1+\epsilon)\log \log x$.  
If $ga$, $gb$, $ha$ and $hb$ are all in ${\mathcal S}$ then we must have $2K - \Omega(g)-\Omega(a)-\Omega(b)-\Omega(h) \ge 0$, so that we may bound the quantity in \eqref{3.2} by 
		\begin{equation} \label{3.3} 
		\ll \sum_{\substack{ g, h \\ 1/\delta \ll g, h \ll \sqrt{x}}}\ \  \sum_{\substack{a, b \\ ga, gb, ha, hb \in [x,x+y] } } 2^{2K-\Omega(g)-\Omega(a)-\Omega(b)-\Omega(h)}. 
		\end{equation} 
Here the weight $2^{2K-\Omega(gabh)}$ was chosen with the benefit of hindsight, starting with $\lambda^{2K-\Omega(gabh)}$ for $\lambda\ge 1$ and optimizing the value of $\lambda$.

		Noting that $(1+\delta)^{-1} \le g/h \le (1+\delta)$, and invoking Lemma \ref{Shiu} we may bound the quantity in \eqref{3.3} by 
		\begin{align} \label{3.4}
		&\ll 2^{2K} \sum_{\substack{ 1/\delta \ll g, h \ll \sqrt{x} \\ (1+\delta)^{-1} \le g/h\le (1+\delta)} } 2^{-\Omega(g)-\Omega(h)} 
		\Big( \sum_{a \in [x/g,(x+y)/g]} 2^{-\Omega(a)}\Big)^2 \nonumber\\
		& \ll 2^{2K}  \sum_{\substack{ 1/\delta \ll g, h \ll \sqrt{x} \\ (1+\delta)^{-1} \le g/h\le (1+\delta)} } 2^{-\Omega(g)-\Omega(h)}  \Big( \frac{y}{g\log x} \exp\Big( \sum_{p\le x} \frac{1}{2p}\Big)\Big)^2 \nonumber\\ 
		&\ll \frac{2^{2K}y^2  }{\log x}  \sum_{\substack{ 1/\delta \ll g, h \ll \sqrt{x} \\ (1+\delta)^{-1} \le g/h\le (1+\delta)} } \frac{2^{-\Omega(g)-\Omega(h)} }{g^2}. 
		\end{align} 
		
		It remains to bound the sums over $g$ and $h$ in \eqref{3.4}. To this end, we split the sum over $g$ into dyadic blocks $G\le g\le 2G$ with $1/\delta \ll G \ll \sqrt{x}$.  In the range $1/\delta \ll G \le 1/\delta^2$ note that 
		$$ 
		\sum_{\substack{ G \le g \le 2G \\ (1+\delta)^{-1} \le g/h\le (1+\delta)} } \frac{2^{-\Omega(g)-\Omega(h)} }{g^2} \ll 
		\frac{1}{G^2} \sum_{G\le g\le 2G} \sum_{ (1+\delta)^{-1} g \le h \le (1+\delta)g} 1 \ll \frac{1}{G^2} G (G\delta) \ll \delta. 
		$$ 
		In the range $1/\delta^2 < G \ll \sqrt{x}$, using Lemma \ref{Shiu} twice we obtain the bound 
		$$ 
		\sum_{\substack{ G\le g \le 2G \\ (1+\delta)^{-1} \le g/h\le (1+\delta)} } \frac{2^{-\Omega(g)-\Omega(h)} }{g^2} \ll  
		\frac{1}{G^2} \sum_{G\le g\le 2G} 2^{-\Omega(g)} \frac{\delta G}{(\log G)^{\frac 12}} \ll \frac{\delta}{\log G}.  
		$$ 
		Splitting the interval $1/\delta \ll g \ll \sqrt{x}$ into dyadic blocks, and using  the above two estimates, we conclude that the sums over $g$ and $h$ in \eqref{3.4} contribute $ \ll \delta \log (1/\delta) + \delta \log \log x \ll \delta \log \log x$.  Inserting this in \eqref{3.4}, we conclude that the number of off-diagonal solutions is 
		\begin{equation}\label{eqn: numbsolun}
		\ll \frac{2^{2K} y^2}{\log x} \delta \log \log x 
		\end{equation}
		which is $o(y^{2})$  in the range $y \le x/(\log x)^{2\log 2 -1 +2\epsilon}$, upon recalling that $K= (1+\epsilon) \log \log x$.  
	\end{proof}

	\begin{remark}\label{rem: 4.3}
		Proposition~\ref{prop:typical count} also answers the following question: what is the largest $y$ such that the product set of $\CA=[x,x+y]$ has size $|{\mathcal A} \cdot {\mathcal A}| \gg |{\mathcal A}|^2$?  Since the Cauchy-Schwarz inequality gives $|{\mathcal A} \cdot {\mathcal A}| \ge |{\mathcal S} \cdot {\mathcal S}| \ge 
		 |{\mathcal S}|^4/ E_{\times}({\mathcal S})$, from Proposition~\ref{prop:typical count} we find that the product set of $[x,x+y]$ has its maximal size $\sim y^2/2$ 
		 in the range $y\le x/(\log x)^{2\log 2 -1+\epsilon}$.  On the other hand, if $y$ is larger than $x/(\log x)^{2\log 2 -1 -\epsilon}$ then apart from $o(y^2)$ exceptions, an element in the product set $[x,x+y] \cdot [x,x+y]$ would be in $[x^2, x^2 +2xy+y^2]$ and have about $2\log \log x$ prime factors, and an application of Selberg's work \cite{Selberg} 
shows that there are at most $o(y^2)$ such elements.  Thus the largest $y$ in this problem is of size $x/(\log x)^{2\log 2 -1 +o(1)}$.  
There are other closely related problems where the same threshold arises; for instance see \cite{PS1990} for work on product sets of dense subsets of the first $N$ integers
(and a random version is studied in \cite{mastrostefano2020maximal}), and see \cite{xuzhou} for a study of  product sets of arithmetic progressions. 
	\end{remark}
		
\begin{remark}  The recent paper \cite{PWX} studies high moments of random multiplicative functions over short intervals $[x, x+y]$, and produces a range of $y$ where the high moments match the Gaussian moments (establishing a central limit theorem).   The valid range for $y$ there is weaker than what we establish in Corollary~\ref{thm: 2log2-1 stein}, 
and only when $x/y$ is larger than an arbitrarily large power of $\log x$ does the method of moments yield a central limit theorem.  

The flexibility of restricting to a dense subset ${\mathcal S}$ of $[x,x+y]$ can facilitate the computation of some higher moments.  Indeed the key point in our argument  is 
that, when restricted to integers with a typical number of prime factors,  the fourth moment matches that of a Gaussian so long as $y\le x/(\log x)^{2\log 2 -1+\epsilon}$.   The argument in Remark~\ref{rem: 4.3} shows that the fourth moment blows up if $y \ge x/(\log x)^{2\log 2-1 -\epsilon}$.  Even when restricted to integers with a typical number of prime factors, higher moments will still blow up, so that Corollary~\ref{thm: 2log2-1 stein} is not accessible by the method of moments.  

To illustrate briefly, consider the range of $y$ for which the sixth moment blows up.   Let ${\mathcal S}$ be any dense subset of $[x,x+y]$, and let ${\mathcal S_0}$ denote the elements in ${\mathcal S}$ with $\Omega(n) =(1+o(1)) \log \log x$ so that 
$|{\mathcal S_0}|$ is also $\sim y$.  Then, an application of the Cauchy-Schwarz inequality gives,  
$$ 
\E\Big[ \Big|\sum_{n\in {\mathcal S}} f(n) \Big|^6 \Big] \ge \E\Big[ \Big| \sum_{n\in {\mathcal S_0}} f(n) \Big|^6 \Big]  \ge \frac{|{\mathcal S_0}|^6}{|{\mathcal S}_0 \cdot {\mathcal S_0} 
\cdot {\mathcal S_0}|}. 
$$ 
Now the triple product set ${\mathcal S}_0 \cdot {\mathcal S}_0 \cdot {\mathcal S}_0$ is a subset of the integers in $[x^3, (x+y)^3]$ having $(3+o(1))\log \log x$ prime factors, 
and this set has size $yx^2 (\log x)^{2-3\log 3+o(1)}$.  Therefore 
$$
\E\Big[ \Big|\sum_{n\in {\mathcal S}} f(n) \Big|^6 \Big]  \ge \frac{y^5}{x^2} (\log x)^{3\log 3- 2+o(1)}, 
$$ 
and this is much bigger than $y^3$ if $y \ge x/(\log x)^{\frac 32\log 3 -1 -\epsilon}$.  Note that $\frac 32 \log 3-1 =0.6479 \ldots$, while $2\log 2 -1 = 0.3862 \ldots$. 

\end{remark}

\section{Proof of  Corollary~\ref{thm: sum of two squares}}

Let $\CA$ denote the set of integers in $[x,x+y]$ that are the sum of two squares, where we assume that $x$ and $y$ are large with $x^{\frac 13} \le y = o(x)$.  
The lower bound on $y$ ensures, by work of Hooley \cite{Hooley}, that $|\CA| \asymp y/\sqrt{\log x}$.  As with the proof of Corollary~\ref{thm: 2log2-1 stein}, we shall apply Corollary~\ref{cor: indicator} with a suitable choice of ${\mathcal S} \subset {\mathcal A}$.  Note that the third condition required in Corollary~\ref{cor: indicator} follows from Lemma~\ref{lem3.1: large prime}, and it remains to specify ${\mathcal S}$ and verify the first two conditions there.   As in Section 4, we write $y=\delta x$.  

Consider first the range $x^{\frac 13} \le y \le x/(\log x)^3$, where we shall simply take ${\mathcal S}= {\mathcal A}$.  Thus the first condition in Corollary~\ref{cor: indicator} is immediate, and it remains to bound the number of non-diagonal solutions to $m_1 m_2 = n_1 n_2$ with $m_1$, $m_2$, $n_1$, $n_2$ in ${\mathcal A}$.   The argument leading up to \eqref{4.3} shows that the number of non-diagonal solutions with $m_1$, $m_2$, $n_1$, $n_2$ in $[x,x+y]$ (ignoring that they are sums of two squares) is 
$\ll \delta y^2 \log x$.  Since $|\CA| \gg y/\sqrt{\log x}$, in the range $\delta\le (\log x)^{-3}$ we see that this bound is $\ll |\CA|^2/\log x$, which verifies condition 2. 

Therefore we may assume that $y$ is in the range $x/(\log x)^3 \le y = o(x)$.  In this range we require a more careful choice of the set ${\mathcal S}$.  Let $a(n)$ denote the indicator function of the set of integers that are sums of two squares.  Recall that $a(n)$ is a multiplicative function with $a(p^k)=1$ if $p$ is $2$ or $p\equiv 1 \bmod 4$, and for $p\equiv 3\bmod 4$ given by $a(p^{2k})=1$ and $a(p^{2k+1}) =0$.  A typical integer $n$ of size $x$ that is a sum of two squares will have about $\frac 12 \log \log x$ prime factors, and indeed such an integer will have about $\frac 12 k$ prime factors below $e^{e^{k}}$.  Our set ${\mathcal S}$ will consist of such typical sums of two squares.

More precisely, let $\epsilon >0$ be small, and let ${\mathcal S}$ be the subset of integers $n \in {\mathcal A}$ satisfying 
$\Omega(n; e^{e^k}) \le (\frac 12+\epsilon) k$ for each natural number $k$ in the range $1/\delta \le e^{e^{k}} \le x$.  Here $\Omega(n;t)$ 
counts the number of prime powers $p^a$ dividing $n$ with $p \le t$.  We begin by showing that $|{\mathcal A}\backslash{\mathcal S}|$ is small 
for small $\delta$, which would verify condition 1 of Corollary~\ref{cor: indicator}.

Let $k$ be a given integer in the range $\log \log (1/\delta) \le k \le \log \log x$; note that since $\delta =o(1)$, we know that $k$ is large (tending to infinity with $x$).  We first 
bound the number of integers $n$ in ${\mathcal A}$ that have $\Omega(n;e^{e^{k}}) \ge (\frac 12+\epsilon) k$.  We apply Shiu's result Lemma~\ref{Shiu} taking there 
$f(n) =\exp(k^{-\frac 12}\Omega(n;e^{e^k})) a(n)$ to obtain 
$$ 
\#\{ n\in {\mathcal A}: \ \ \Omega(n;e^{e^{k}}) \ge (\tfrac 12+\epsilon) k\} \le e^{-(\frac 12+\epsilon) \sqrt{k}} \sum_{x \le n \le x+y} f(n) \ll |{\mathcal A}| \exp(-\epsilon \sqrt{k}).
$$
Summing this over all $k$ in the range $\log \log (1/\delta) \le k \le \log \log x$, it follows that $|{\mathcal A}\backslash {\mathcal S}| = o(x)$, as desired.

Having verified condition 1 of Corollary~\ref{cor: indicator}, it remains lastly to check condition 2; namely to check that there are few non-diagonal solutions to 
$m_1m_2 = n_1 n_2$ with $m_1$, $m_2$, $n_1$, $n_2 \in {\mathcal S}$.  We parametrize solutions as in Lemma~\ref{lem3.2: parameter} writing $m_1=ga$, $m_2=hb$, $n_1=gb$, and $n_2 = ha$, where $(a,b)=1$.  We may assume that $g\neq h$ and $a \neq b$ (since we only want non-diagonal solutions).  As in 
\eqref{3.1} we must have $g/h$ and $a/b$ lying in the interval $[(1+\delta)^{-1}, 1+\delta]$, so that we may assume that $g$, $h$, $a$, $b$ are all $\gg 1/\delta$.  Since $a$ and $b$ are coprime, and $ga$ and $gb$ are both sums of 
two squares, it follows that $g$, $a$, and $b$ must all be sums of two squares; similarly $h$ must also be a sum of two squares. 
Thus, we may suppose that $g$, $h$, $a$, and $b$ are all $\gg 1/\delta$, are all sums of two squares, and as in 
\eqref{3.2} our task is to bound 
\begin{equation} 
\label{5.1} 
\sum_{\substack{ g, h \\ 1/\delta \ll g, h\le \sqrt{x}} } \ \ \sum_{\substack{ a, b \\ ga, gb, ha, hb \in {\mathcal S}}} a(g)a(h)a(a)a(b). 
\end{equation} 
Above we omitted the condition $(a,b)=1$; noted that $\max(g,h)$ or $\max(a,b)$ must be $\le \sqrt{x}$, and assumed that the former condition holds by symmetry.

Break the sum over $g$ in \eqref{5.1} into dyadic blocks $G< g\le 2G$ where $1/\delta \ll G \le \sqrt{x}$.  We wish to estimate the contribution to \eqref{5.1} arising from such a dyadic block.  Select $k$ to be the least integer with $e^{e^k} \ge \max(1/\delta, G)$. Since $ga$, $gb$, $ha$, $hb$ are all in ${\mathcal S}$ we must have 
$\Omega(g;e^{e^{k}}) + \Omega(h;e^{e^{k}}) + \Omega(a; e^{e^{k}}) + \Omega(b;e^{e^k}) \le (1+2\epsilon) k$.  Therefore the contribution of this dyadic block is 
$$ 
\ll 2^{(1+2\epsilon)k} \sum_{\substack{ G < g\le 2G \\ h \in (g/(1+\delta), g(1+\delta)) }} a(g) a(h) 2^{-\Omega(g; e^{e^k}) - \Omega(h;e^{e^k})} 
\sum_{\substack{ a,b \in (x/g, (x+y)/g)} } a(a)a(b) 2^{-\Omega(a;e^{e^k}) -\Omega(b;e^{e^k})}. 
$$
Now using Lemma~\ref{Shiu}, the sum over $a$ above may be bounded by 
$$ 
\ll \frac{y}{g\log x} \exp\Big( \sum_{\substack {p\le e^{e^k} \\ p\equiv 1 \bmod 4 }} \frac{1}{2p} + \sum_{\substack{ e^{e^k} < p \le x \\ p\equiv 1 \bmod 4}} \frac 1p \Big) 
\ll \frac{y}{g\sqrt{\log x}} e^{-\frac k4}. 
$$
Naturally the same bound holds for the sum over $b$, and we conclude that the contribution of the dyadic block $G \le g \le 2G$ is 
\begin{equation} 
\label{5.2}  
\ll e^{k( (1+2\epsilon) \log 2 -1/2)} \frac{y^2}{G^2 \log x}  \sum_{\substack{ G < g\le 2G \\ h \in (g/(1+\delta), g(1+\delta)) }} a(g) a(h) 2^{-\Omega(g; e^{e^k}) - \Omega(h;e^{e^k})} . 
\end{equation}

Consider first the case when $1/\delta \ll G \le 1/\delta^2$.   Here we bound $a(g) a(h) 2^{-\Omega(g; e^{e^k}) - \Omega(h;e^{e^k})}$ by $1$, and note that given $g$ there are 
$\ll \delta g $ choices for $h$.  Noting also that $e^k$ is of size $\log G$, we conclude that in this range of $G$, the quantity in \eqref{5.2} may be bounded by 
\begin{equation} 
\label{5.3} 
\ll \delta (\log G)^{\frac 14} \frac{y^2}{\log x}. 
\end{equation} 

 Now consider the range $1/\delta^2 \le G\le \sqrt{x}$.  Here we may use Lemma~\ref{Shiu} twice to bound the quantity in \eqref{5.2} by 
 \begin{align} 
 \label{5.4} 
& \ll (\log G)^{(1+2\epsilon)\log 2-1/2} \frac{y^2}{G^2 \log x} \sum_{G\le g\le 2G} a(g) 2^{-\Omega(g; e^{e^k})}  
\frac{\delta g}{\log G} (\log G)^{\frac 14} \nonumber \\ 
&\ll \frac{y^2}{\log x} \frac{\delta}{(\log G)^{2-(1+2\epsilon)\log 2}}. 
\end{align}

We now return to the problem of bounding \eqref{5.1}.  Using \eqref{5.3}  the contribution of the dyadic blocks with $1/\delta \ll G\le 1/\delta^2$ may 
be bounded by $\ll \delta (\log 1/\delta)^{\frac 54} y^2/\log x$ (since there are $\ll \log (1/\delta)$ such dyadic blocks).  Using \eqref{5.4}, the 
contribution of all the dyadic blocks with $1/\delta^2 \le G\le \sqrt{x}$ is $\ll \delta y^2/\log x$ --- the key fact here is that $2-(1+2\epsilon)\log 2 >1$ 
(for suitably small $\epsilon$) so that when $(\log G)^{-(2-(1+2\epsilon)\log 2)}$ is summed over the powers of $2$ in this range, the resulting sum is 
$\ll 1$.  We conclude that \eqref{5.1}, which bounds the non-diagonal solutions to $m_1m_2 =n_1 n_2$, may be bounded by 
$$ 
\ll \delta (\log 1/\delta)^{\frac 54} \frac{y^2}{\log x} = o(|{\mathcal A}|^2). 
$$ 
This completes our verification of condition 2 in Corollary~\ref{cor: indicator}, and thus our proof of Corollary~\ref{thm: sum of two squares}.

 We remark that the proof goes through for more general sifted sets ${\mathcal A}$.  For instance, suppose ${\mathcal A}$ is 
 the set of integers composed of primes lying in subset of the primes with relative density $\rho$.  Then so long as $\rho \le 1/\log 4 -\epsilon$, 
and $y=o(x)$ is such that $[x,x+y]$ contains the expected number of elements of ${\mathcal A}$ (which is about $y/(\log x)^{1-\rho}$), then one can obtain a 
suitable central limit theorem.

\section{Shifted primes: Proof of Corollary~\ref{thm: shift primes}}
To deduce Corollary~\ref{thm: shift primes} from Theorem~\ref{thm1.1}, we need only show that the number 
of non-diagonal solutions to $(p+k)(q+k) = (r+k)(s+k)$ (with $p$, $q$, $r$ and $s$ being primes below $N$) is 
$o(\pi(N)^2)$.  We use the parametrization of Lemma~\ref{lem3.2: parameter} to write 
$p+k= ga$, $q+k = hb$, $r+k = gb$ and $s+k= ha$, with $a\neq b$ and $g \neq h$ (since we are only interested in bounding 
non-diagonal solutions).    Thus, with $\p$ denoting the indicator function of primes, we must bound 
$$ 
\sum_{ a \ne b } \sum_{ \substack { g\neq h \\ ga, gb, ha, hb \le N+k} } \p (ga-k) \p(gb-k) \p(ha-k) \p (hb-k). 
$$

Note that either $\max(a,b)$ or $\max(g,h)$  must  be $\le \sqrt{N+k}$.  By symmetry, we 
may restrict attention to the former case, and also assume that $a < b$.  Thus the non-diagonal solutions are bounded by
$$ 
\ll \sum_{a < b \le \sqrt{N+k}} \ \ \sum_{g, h\le (N+k)/b} \p (ga-k) \p(gb-k) \p(ha-k) \p(hb-k). 
$$ 
For each small prime $p$ with $p \nmid kab(b-a)$, $g$ must avoid two distinct residue classes $\bmod \ p$ (namely the residue classes $k/a$ and $k/b \bmod p$)  
in order for $ga-k$ and $gb-k$ to be prime.  For the primes $p$ dividing $kab (b-a)$ ignore any constraints that $g$ must satisfy $\bmod \ p$.     Then, a straight-forward upper bound sieve gives 
$$ 
\sum_{g\le (N+k)/b} \p (ga -k) \p(gb-k) \ll \frac{N}{b (\log N)^2} \Big( \frac{|k|ab(b-a)}{\phi(|k|ab(b-a))}\Big)^2  
$$
 and of course the same holds for the sum over $h$.  We conclude that the off-diagonal solutions are bounded by 
 \begin{align*}
& \ll \sum_{a< b \le \sqrt{N+k}} \frac{N^2}{b^2 (\log N)^4} \Big( \frac{|k|ab(b-a)}{\phi(|k|ab(b-a))}\Big)^4\\
& \ll \frac{N^2}{(\log N)^4} \Big(\frac{|k|}{\phi(|k|)}\Big)^4 
 \sum_{a<b \le \sqrt{N+k}} \frac{1}{b^2} \Big( \Big( \frac{a}{\phi(a)}\Big)^{12} + \Big( \frac{b}{\phi(b)}\Big)^{12} + \Big( \frac{b-a}{\phi(b-a)}\Big)^{12} \Big), 
\end{align*} 
upon using the AM-GM inequality.  Since $\sum_{n\le x} (n/\phi(n))^{12} \ll x$, it readily follows that the above is 
$$ 
\ll \frac{N^2}{(\log N)^3} \Big( \frac{|k|}{\phi(|k|)}\Big)^4, 
$$ 
 which suffices since $k$ is a fixed non-zero integer.

\section{Proof of Corollary~\ref{thm:large subset} }

This section records the largest subset ${\mathcal A} \subseteq [1,N]$ that we know for which $\sum_{n\in {\mathcal A}} f(n)$ has a 
Gaussian distribution.  The construction is essentially due to Ford \cite{Ford18}, who showed that there is a subset ${\mathcal B} \subseteq [1,N]$ with 
$|{\mathcal B}| \ge N(\log N)^{-\theta} (\log \log N)^{-\frac 32}$ such that $E_{\times}({\mathcal B}) \ll |{\mathcal B}|^2 (\log \log N)^4$ (see \cite[Lemma 3.1, 3.2]{Ford18}).   

Let  ${\mathcal A}$ range uniformly over all subsets of ${\mathcal B}$ with $\lfloor \rho |{\mathcal B}|\rfloor$ elements, where $\rho =(\log \log N)^{-5}$.  Let us compute the 
average number of non-diagonal solutions to $m_1 m_2 = n_1 n_2$ with $m_1$, $m_2$, $n_1$, $n_2 \in {\mathcal A}$.    Note that any such non-diagonal solution must 
arise from a non-diagonal solution to $m_1 m_2 = n_1 n_2$ with $m_1$, $m_2$, $n_1$, $n_2 \in {\mathcal B}$.  Since at least three of $m_1$, $m_2$, $n_1$, $n_2$ 
must be distinct, such a non-diagonal solution would count as a non-diagonal solution in ${\mathcal A}$ with ``probability" $\ll \rho^3$ (three elements of ${\mathcal A}$ are 
specified, and there are $\binom{|{\mathcal B}|-3}{|{\mathcal A}|-3}$ ways of choosing the remaining elements).  It follows that the average number 
of non-diagonal solutions in ${\mathcal A}$ is $\ll \rho^3 E_\times({\mathcal B}) \ll |{\mathcal A}|^2 (\log \log N)^{-1}$.  

We deduce that there exists a subset ${\mathcal A}$ of $[1,N]$ with $|{\mathcal A}| \ge N(\log N)^{-\theta} (\log \log N)^{-7}$ such that the number of 
non-diagonal solutions to $m_1 m_2 =n_1 n_2$ with $m_1$, $m_2$, $n_1$, $n_2 \in {\mathcal A}$ being at most $|{\mathcal A}|^2 (\log \log N)^{-1}$.  Corollary~\ref{thm:large subset}
now follows from Theorem~\ref{thm1.1}.

\section{Proof of Theorem~\ref{thm: Diophantine}}	

In this section we study the distribution of random multiplicative functions twisted by $e(n\alpha)$.  A key input in understanding such sums is the following result of 
Montgomery and Vaughan.  

 \begin{lemma}[Montgomery-Vaughan \cite{MV77}]\label{MVT}  Let $g$ be a multiplicative function with $|g(n)|\le 1$ for all $n$.   Let $x$ be large, and let $\alpha$ be a real number.  Suppose $\alpha$ has a rational approximation $u/v$ such that $|\alpha-u/v|\le 1/v^2$, where $(u,v)=1$ and $v$ lies in the interval $R \le v \le x/R$ for some 
 parameter $R\ge 2$.  
 Then 
    \[
    \sum_{n\le x} g(n) e(n\alpha) \ll \frac{x}{\log x} +\frac{x}{\sqrt{R}} (\log R)^{3/2}.  
    \]
\end{lemma}


To prove Theorem~\ref{thm: Diophantine} we shall apply Theorem~\ref{thm: a_n}, taking $\CA=\CS$ be the set of all positive integers up to $x$ and $a_n = e(n\theta)$.  The variance $V$ equals $\lfloor x\rfloor$, and we must check the three criteria given in Theorem~\ref{thm: a_n}.    
The first condition holds automatically since ${\mathcal S}={\mathcal A}$.  We now check the third condition, and then consider the second condition (which requires the most work).  
The third condition in Theorem~\ref{thm: a_n} requires a good bound for 
\begin{equation} 
\label{8.1} 
\Big|\sum_{\substack{m_1, m_2, n_1, n_2 \le x \\ m_1 m_2 = n_1 n_2 \\ P(m_1)=P(n_1)=P(m_2)=P(n_2)}} e(\theta(m_1+m_2 -n_1-n_2)) \Big| \le 
\sum_{\substack{m_1, m_2 \le x \\ P(m_1)=P(m_2) } } d(m_1m_2).
\end{equation} 
Since $d(m_1m_2) \le d(m_1)d(m_2) \le \frac 12(d(m_1)^2+ d(m_2)^2)$, we may bound the above quantity by 
$$ 
 \le \sum_{m_1\le x} d(m_1)^2 \sum_{\substack{ m_2 \le x \\ P(m_2) = P(m_1)}} 1.
$$
Given $m_1$, arguing as in Lemma~\ref{lem3.1: large prime} we may bound the inner sum over $m_2$ by $\ll x \exp(-\frac 12 \sqrt{\log x \log \log x}) \ll x/(\log x)^{10}$, 
so that the quantity in \eqref{8.1} may be bounded by 
$$ 
\ll \frac{x}{(\log x)^{10}} \sum_{m_1 \le x} d(m_1)^2 \ll \frac{x^2}{(\log x)^{7}},
$$
 which is more than we need.

It remains to verify the second condition, which requires us to bound 
\begin{equation*}\label{eqn: dio}
    \sum_{\substack{m_1, m_2, n_1, n_2 \le x\\m_1m_2 = n_1n_2\\
	P(m_1) =P(n_1) \\
	P(m_2) = P(n_2)\\
	m_1 \neq n_1, m_2 \neq n_2}} e((m_1+m_2- n_1 -n_2 )\theta).
\end{equation*}
We use the parametrization in Lemma~\ref{lem3.2: parameter}, to write $m_1=ga$, $m_2 =hb$, $n_1= gb$ and $n_2=ha$.  The 
constraints on $m_1$, $m_2$, $n_1$, $n_2$ then become $(a,b)=1$ with $a\neq b$,  $P(ab) \le \min ( P(g), P(h))$, and $\max(a,b) \times \max(g,h)\le x$.  
Thus the sum we wish to bound becomes 
\begin{equation} 
\label{8.2} 
\sum_{\substack{ \max (a,b ) \times \max(g, h) \le x \\ a\neq b, \ (a,b)=1 \\ P(ab)\le \min (P(g), P(h))}} e((g-h)(a-b) \theta). 
\end{equation} 
Since $\max(a,b) \times \max(g,h) \le x$, we may break the sum above into the cases (1) when $\max(g,h) \le \sqrt{x}$, (2) when $\max(a,b )\le \sqrt{x}$, taking 
care to subtract the terms satisfying (3) $\max(a,b)$ and $\max(g,h)$ both below $\sqrt{x}$.  

Before turning to these cases, we record a preliminary lemma which will be useful in our analysis. 

\begin{lemma} \label{lem8.2} Let $\theta$ be an irrational number satisfying the Diophantine condition \eqref{1.4}. 
Let ${\mathcal L} = {\mathcal L}(x)$ denote the set of all integers $\ell$ with $|\ell| \le \sqrt{x}$ such that for some 
$v \le (\log x)^5$ one has $\Vert v\ell \theta \Vert \le x^{-\frac 13}$.  Then $0$ is in ${\mathcal L}$, and for any two distinct elements $\ell_1$, $\ell_2 \in {\mathcal L}$ 
we have $|\ell_1 - \ell_2| \gg (\log x)^{5}$.
\end{lemma} 
\begin{proof}  Evidently $0$ is in ${\mathcal L}$, and the main point is the spacing condition satisfied by elements of ${\mathcal L}$.  If $\ell_1$ and $\ell_2$ are distinct elements of ${\mathcal L}$ then there exist $v_1$, $v_2 \le (\log x)^5$ with $\Vert v_1 \ell_1 \theta\Vert \le x^{-\frac 13}$ and $\Vert v_2 \ell_2 \theta \Vert \le x^{-\frac 13}$.  It follows that 
$\Vert v_1 v_2 (\ell_1 -\ell_2) \theta\Vert \le 2(\log x)^5 x^{-\frac 13}$.  The desired bound on $|\ell_1 -\ell_2|$ now follows from the Diophantine property that we required of $\theta$, namely that $\Vert q\theta \Vert \gg \exp(-q^{\frac 1{50}})$.  
\end{proof}

\subsection{Case 1: $\max (g, h )\le \sqrt{x}$} Suppose that $g$ and $h$ are given with $g$ and $h$ below $\sqrt{x}$, and consider the sum over $a$ and $b$ in \eqref{8.2}.  
We distinguish two sub-cases, depending on whether $g-h$ lies in ${\mathcal L}$ or not.  Consider first the situation when $g-h \not\in {\mathcal L}$.  Using M{\" o}bius inversion to detect the condition that $(a,b)=1$, the sums over $a$ and $b$ may be expressed as (the $O(1)$ error term accounts for the term $a=b=1$ which must be omitted)
\begin{align}\label{8.3}  
&\sum_{\substack{ k\le x/\max(g, h)\\ P(k) \le \min(P(g),P(h))}} \mu(k) \sum_{\substack{ r, s \le x/(k\max (g,h)) \\ P(r), P(s) \le \min(P(g),P(h)) }} e(k(g-h)(r-s)\theta) +O(1) 
\nonumber \\
&= \sum_{\substack{ k\le x/\max(g, h)\\ P(k) \le \min(P(g),P(h))}} \mu(k) \Big|\sum_{\substack{ r \le x/(k\max (g,h)) \\ P(r) \le \min(P(g),P(h)) }} e(k(g-h)r\theta)\Big|^2 +O(1). 
\end{align}
If $k> (\log x)^2$ then we bound the sum over $r$ above by $x/(k\max(g,h))$, and so these terms contribute to \eqref{8.3} an amount 
$$
\ll \sum_{k>(\log x)^2} \frac{x^2}{k^2 \max(g,h)^2} \ll \frac{x^2}{(\log x)^2 \max (g, h)^2}.
$$
Now consider $k\le (\log x)^2$, and find (using Dirichlet's theorem) a rational approximation $u/v$ to $k(g-h) \theta$ with $|k(g-h) \theta -u/v| \le 1/(vx^{\frac 13})$ and $v\le x^{\frac 13}$.  
Since $g-h \not \in {\mathcal L}$ by assumption, it follows that $v \ge (\log x)^3$, and therefore an application of Lemma~\ref{MVT} shows that the sum over $r$ in \eqref{8.3} 
is $\ll x/(k\max(g,h)\log x)$.  Thus the terms $k \le(\log x)^2$ contribute to \eqref{8.3} an amount bounded by 
$$ 
\sum_{k\le (\log x)^2} \frac{x^2}{k^2 \max(g,h)^2 (\log x)^2} \ll \frac{x^2}{(\log x)^2 \max(g,h)^2}. 
$$ 
Summing this over all $g$, $h \le \sqrt{x}$, we conclude that the contribution of terms with $\max(g,h)\le \sqrt{x}$ and  $g-h \not\in {\mathcal L}$ to \eqref{8.2} is 
$$
\ll \sum_{g, h \le \sqrt{x}} \frac{x^2}{(\log x)^2 \max(g,h)^2} \ll \frac{x^2}{\log x}. 
$$ 

Now consider the contribution of the terms $\max(g,h)\le \sqrt{x}$ where $g-h$ lies in ${\mathcal L}$.   Note that in \eqref{8.2} we allow for the possibility that $g=h$; we begin by 
estimating these terms (which could also be handled as in our argument for the terms in \eqref{8.1}).  The terms $g=h \le \sqrt{x}$ give 
$$ 
\le \sum_{g\le \sqrt{x}} \Big(\sum_{\substack{a\le x/g \\ P(a) \le P(g) }} 1 \Big)^2  \le \sum_{(\log x)^2 \le g \le \sqrt{x}} \frac{x^2}{g^2} + \sum_{g\le (\log x)^2} \Psi(x/g,(\log x)^2) 
\ll \frac{x^2}{\log x}. 
$$
Now consider the terms with $g-h \in {\mathcal L}$ with $g-h \neq 0$.  Bounding the sum over $a$ and $b$ trivially by $\le (x/\max(g,h))^2$, we see that the 
contribution of these terms is 
$$ 
\ll \sum_{\substack{g \neq h \le \sqrt{x} \\ g- h \in {\mathcal L}}} \frac{x^2}{\max(g,h)^2} \ll \sum_{g\le \sqrt{x}} \frac{x^2}{g^2} \sum_{\substack{ h< g\\ g-h \in {\mathcal L}}} 1 
\ll  \sum_{g\le \sqrt{x}} \frac{x^2}{g^2} \frac{g}{(\log x)^5} \ll \frac{x^2}{(\log x)^4}, 
$$ 
where we used Lemma \ref{lem8.2} to bound the sum over $h$.

We conclude that the contribution of terms with $\max(g,h)\le \sqrt{x}$ to \eqref{8.2} is $\ll x^2/\log x$, completing our discussion of this case.  
 
 \subsection{Case 2: $\max\{a, b\}\le \sqrt{x}$}  Here we must bound 
 $$ 
 \sum_{\substack{ a\neq b \le \sqrt{x} \\ (a, b)=1}} \Big| \sum_{\substack{g\le x/\max(a,b) \\ P(ab) \le P(g)} } e(g(a-b)\theta)\Big|^2.
 $$
 Again we distinguish the cases when $a-b\in{\mathcal L}$, and when $a-b \not \in {\mathcal L}$.  In the first case, we bound the sum over $g$ above trivially by $\le x/\max(a,b)$, 
 and thus these terms contribute (using Lemma \ref{lem8.2})
 $$ 
 \ll \sum_{a\le \sqrt{x}} \frac{x^2}{a^2} \sum_{\substack{ b <a \\ a-b\in {\mathcal L}}} 1 \ll \sum_{a\le \sqrt{x}} \frac{x^2}{a^2} \frac{a}{(\log x)^5} \ll \frac{x^2}{(\log x)^4}. 
 $$ 
 
 Now consider the case when $a-b\not\in {\mathcal L}$.  Using Dirichlet's theorem we may find a rational approximation $u/v$ to $(a-b)\theta$  such that 
 $|(a-b)\theta-u/v| \le 1/(vx^{\frac 13})$ and $v\le x^{\frac 13}$.  Since $(a-b)\not \in {\mathcal L}$, it follows that $v \ge (\log x)^5$.  Therefore, two applications of Lemma \ref{MVT}  give
 $$ 
 \sum_{\substack{g\le x/\max(a,b) \\ P(g) \ge P(ab)}} e(g(a-b)\theta) = \sum_{\substack{g\le x/\max(a,b)}} e(g(a-b)\theta) - \sum_{\substack{g\le x/\max(a,b) \\ P(g) <P(ab)} } 
 e(g(a-b)\theta) \ll \frac{x}{\max(a,b) \log x}.
 $$ 
 Thus the contribution of the terms with $a-b\not \in {\mathcal L}$ is 
 $$ 
 \ll \sum_{a, b\le \sqrt{x}} \frac{x^2}{(\log x)^2 \max(a,b)^2} \ll \frac{x^2}{\log x}. 
 $$ 
 
 Thus the contribution to \eqref{8.2} from the Case 2 terms is $\ll x^2/\log x$.  
 
 \subsection{Case 3: $\max(a,b)$ and $\max(g,h) \le \sqrt{x}$}  Here we must bound 
 $$ 
 \sum_{\substack{ a\neq b \le \sqrt{x} \\ (a, b)=1}} \Big| \sum_{\substack{g\le \sqrt{x} \\ P(ab) \le P(g)} } e(g(a-b)\theta)\Big|^2,
 $$ 
 and our argument in Case 2 above furnishes the bound $\ll x^2/\log x$.

 Combining our work in the three cases, we conclude that the quantity in \eqref{8.2} is $\ll x^2/\log x$.  This verifies the second condition needed to 
 apply Theorem~\ref{thm: a_n} and completes our proof of Theorem~\ref{thm: Diophantine}.

\section{Rademacher random multiplicative functions}\label{sec: rad}

In this section we briefly indicate the analogues of our results in the Rademacher model of random multiplicative functions, where $f(p) = \pm 1$ with  equal 
probability (and chosen independently for different primes), and $f(n)$ is taken to be $0$ if $n$ has a square factor.  In our work above, a key role was played by the fourth moment, 
which in the Steinhaus case led to solutions to $m_1 m_2 =n_1n_2$ and to the notion of the multiplicative energy.  If we consider the corresponding fourth moment in the Rademacher case, we are led to (with ${\mathcal A}$ denoting a set of square-free integers)
$$ 
\E \Big[ \Big( \sum_{n\in {\mathcal A}} f(n) \Big)^4 \Big]  = \sum_{\substack{ n_1, n_2, n_3, n_4 \in {\mathcal A} \\ n_1 n_2 n_3 n_4 =\square }}1. 
$$ 
 Thus the analogue of the multiplicative energy here is what may be termed the {\sl square energy of ${\mathcal A}$} namely: 
\[
E_{\square}(\mathcal{A}): = \#\{(n_1, n_2, n_3, n_4)\in \mathcal{A}^{4}:  n_1 n_2n_3n_4 = \square  \}.
 \]
 Note that there are $\sim 3 |{\mathcal A}|^2$ diagonal solutions, given by the pairings $n_1=n_2$ and $n_3=n_4$; $n_1= n_3$ and $n_2=n_4$; or $n_1= n_4$ and $n_2=n_3$.  
Taking this difference into account, and arguing as in our proof of Theorem~\ref{thm: a_n} and the simplified Theorem~\ref{thm1.1} we can establish the following result (whose proof we omit).  

\begin{theorem}\label{cor: general rad} 
Let $\mathcal{A}\subset [1, N]$ be a set of square-free integers with 
$$
|\CA|\ge N \exp(-\tfrac 13\sqrt{\log N \log \log N}).
$$
Suppose that there exists a subset $\mathcal{S}\subset \mathcal{A}$ with $|\mathcal{S}|=(1+o(1))|\mathcal{A}|$ and satisfying
\[
E_{\square}(\mathcal{S})   = (3+o(1))|\mathcal{S}|^{2}.
\]
Then as $f$ ranges over Rademacher random multiplicative functions, the quantity
		\[
		\frac{1}{\sqrt{|\mathcal{A}|} } \sum_{n \in \mathcal{A}} f(n)
		 \]
is distributed like a standard (real) normal random variable with mean $0$ and variance $1$. 
\end{theorem}

With suitable modifications to the proofs, the results that we have enunciated for the Steinhaus model would 
extend to the Rademacher case.   We sketch one example, treating the distribution of Rademacher random multiplicative 
functions in short intervals, which is an analogue of Corollary~\ref{thm: 2log2-1 stein} and improves upon the earlier work in \cite{CS}.

	\begin{corollary}\label{thm: 2log2-1 stein-rad}  Let $x$ and $y$ be large, with $x^{1/5}\log x \ll  y\le x/(\log x)^{\alpha -\epsilon}$ where $\alpha= 2\log 2 -1$. Let ${\mathcal T}$ denote the set of all square-free integers in $[x,x+y]$. Then, for a random Rademacher multiplicative function $f$, the quantity 
		$$ 
		\frac{1}{\sqrt{|\mathcal{T}|}} \sum_{ x \le n \le x+y} f(n), 
		$$ 
		is distributed like a standard  normal random variable with mean $0$ and variance $1$. 
	\end{corollary}
	
In \cite{CS} such a result was established in the range $x^{\frac 15} \log x \ll y = o(x/\log x)$.  The lower bound on $y$ is to ensure that the interval $[x,x+y]$ contains 
$\gg y$ square-free integers (which follows from \cite{FT1992}).  Thus we restrict attention to the range $x/(\log x)^2 \le y \le x/(\log x)^{\alpha -\epsilon}$.  In this range we choose ${\mathcal S}$ to be the set of integers $n \in {\mathcal T}$ with $\Omega(n) \le (1+\epsilon) \log \log x$.  Then, as in the proof of Proposition~\ref{prop:typical count}, we have 
$|{\mathcal T} \backslash {\mathcal S}| = o(y)$, and we need only show that the number of non-diagonal solutions to $n_1 n_2 n_3 n_4= \square$ (with $n_i \in {\mathcal S}$) is 
$o(y^2)$.  

Write, as earlier, $y= \delta x$.  Let $r$ denote the gcd of $n_1$ and $n_2$, and let $s$ denote the gcd of $n_3$ and $n_4$.  Write $n_1= r u_1$, $n_2 = r u_2$, and $n_3= s v_1$ and $n_4= s v_2$.   Now $u_1 u_2$ and $v_1 v_2$ are square-free integers (since $(u_1, u_2) = (v_1, v_2)=1$), and their product is a square, which means that $u_1 u_2= v_1 v_2$.    Using our parametrization in Lemma~\ref{lem3.2: parameter}, write $u_1 =ga$, $u_2 = hb$, $v_1 = gb$, $v_2= ha$, where now we also know that $(g,h)=(a,b)=1$.  Thus we have 
parametrized our solutions to $n_1n_2 n_3 n_4 =\square$ as $n_1 =rga$, $n_2 = rhb$, $n_3= sgb$, $n_4= sha$.  If two out of the three possibilities $r=s$, $g=h$, or $a=b$ 
occur, then we obtain diagonal solutions;  therefore we may assume that at most one of the equalities $r=s$, $g=h$, or $a=b$ can occur.

Since the variables $n_i$ must all lie in $[x,x+y]$, it follows that 
$$ 
\frac{n_1 n_2}{n_3 n_4} = \frac{r^2}{s^2} \in [(1+\delta)^{-2}, (1+\delta)^{2}], \qquad \text{ or }\qquad \frac rs \in [(1+\delta)^{-1}, (1+\delta)]. 
$$ 
In particular, either $r=s$, or both $r$ and $s$ are $\gg 1/\delta$.  Similarly we also find that $g/h$ and $a/b$ must lie in $[(1+\delta)^{-1},(1+\delta)]$, 
so that either $g=h=1$ or $g$, $h \gg 1/\delta$, and either $a=b=1$ or $a$, $b \gg 1/\delta$.

{\bf Case 1: $g=h=1$, or $a=b=1$.}  If $g=h=1$, then we must bound the number of non-diagonal solutions  to $n_1 n_3 = n_2 n_4$, and our work in Proposition~\ref{prop:typical count} shows that this is $o(y^2)$.  Similarly if $a=b=1$ then we have non-diagonal solutions to $n_1 n_4=n_2n_3$, and these again are $o(y^2)$. 

{\bf Case 2: $r=s$.}   This is a little different from Case 1, since we may have $r=s$ without both being necessarily $1$.  Suppose first that $r=s \le \sqrt{x}$.  Here note that we 
are counting off-diagonal solutions to $u_1 u_2 = v_1 v_2$, where $u_1$, $u_2$, $v_1$, $v_2$ are in $[x/r, (x+y)/r]$, and $\Omega(u_1)$, $\Omega(u_2)$, $\Omega(v_1)$, $\Omega(v_2)$ are all below $(1+\epsilon) \log \log x$ which is $\le (1+2\epsilon) \log \log (x/r)$. Therefore, Proposition~\ref{prop:typical count} applies here to show that the number of non-diagonal choices for $u_1$, $u_2$, $v_1$, $v_2$ is $o(y^2/r^2)$, and summing over $r \le \sqrt{x}$ produces a bound of $o(y^2)$ for this count. 

Now suppose that $r=s > \sqrt{x}$.  Here we have $\le (y/r +1)$ choices for $u_1$ and $u_2$, and then $v_1$ and $v_2$ are fixed in $O(d(u_1u_2)) = O(x^{\epsilon})$ ways.  
Therefore these terms contribute  
$$ 
\ll \sum_{\sqrt{x} \le r=s \le x} \Big( \frac {y^2}{r^2} +1 \Big) x^{\epsilon} = o(y^2). 
$$ 

{\bf Case 3: $r\neq s$, $g\neq h$, $a\neq b$.}  We now ignore the conditions that $(g,h)=(a,b)=1$, so that the pairs $r$, $s$; $g$, $h$; and $a$, $b$ are now all on an equal footing. 
Since $\Omega(rga)$, $\Omega(rhb)$, $\Omega(sgb)$, $\Omega(sha)$ are all assumed to be $\le K:=(1+\epsilon) \log \log x$, it follows that $\Omega(r)+\Omega(s) + \Omega(g)+ \Omega(h) + \Omega(a)+\Omega(b) \le 2K$.  Without loss of generality we may assume that $\max (r,s ) \ll \max (g, h) \ll \max(a,b)$, so that $\max(r, s) \ll x^{\frac 13}$, and 
$\max(g,h) \ll (x/\max(r, s))^{\frac 12}$. 
Therefore it is enough to bound 
$$ 
\sum_{\substack{1/\delta \ll r, s \ll x^{\frac 13} \\ r/s \in [(1+\delta)^{-1}, (1+\delta)]}} \ \ \sum_{\substack{ 1/\delta \ll g, h \ll (x/\max(r,s))^{\frac 12} \\ g/h \in [(1+\delta)^{-1}, (1+\delta)]} }
\ \ \sum_{\substack{ a \in [x/(rg), (x+y)/(rg)] \\ b\in [x/(sh),(x+y)/(sh)] } } 2^{2K-\Omega(r)-\Omega(s)-\Omega(g)-\Omega(h) - \Omega(a)-\Omega(b)}.
$$
Using Shiu's Lemma~\ref{Shiu} to bound the sums over $a$ and $b$, we see that the above may be bounded by 
$$ 
\ll \frac{2^{2K} y^2}{\log x} \sum_{\substack{1/\delta \ll r, s \ll x^{\frac 13} \\ r/s \in [(1+\delta)^{-1}, (1+\delta)]}}  \frac{2^{-\Omega(r)-\Omega(s)}}{r^2} 
\sum_{\substack{ 1/\delta \ll g, h \ll x^{\frac 12} \\ g/h \in [(1+\delta)^{-1}, (1+\delta)]} }\frac{2^{-\Omega(g) -\Omega(h)}}{g^2}.
$$ 
The sums over $r$, $s$, and $g$, $h$ above are nearly identical to the sums over $g$ and $h$ appearing in \eqref{3.4}, and thus may be bounded 
by $\ll (\delta \log \log x)$ (using our work leading up to \eqref{eqn: numbsolun}).  Thus, we conclude that the number of non-diagonal solutions counted in this case is 
$$ 
\ll \frac{2^{2K} y^2}{\log x} (\delta \log \log x)^2 = o(y^2).  
$$

This completes our proof of Corollary~\ref{thm: 2log2-1 stein-rad}.

	\bibliographystyle{abbrv}
	\bibliography{CLT}{}

\begin{thebibliography}{10}

\bibitem{BNR}
J.~Benatar, A.~Nishry, and B.~Rodgers.
\newblock Moments of polynomials with random multiplicative coefficients.
\newblock {\em Mathematika}, 68(1):191--216, 2022.

\bibitem{BJ2004}
R.~Blei and S.~Janson.
\newblock Rademacher chaos: tail estimates versus limit theorems.
\newblock {\em Ark. Mat.}, 42(1):13--29, 2004.

\bibitem{CS}
S.~Chatterjee and K.~Soundararajan.
\newblock Random multiplicative functions in short intervals.
\newblock {\em Int. Math. Res. Not. IMRN}, (3):479--492, 2012.

\bibitem{FT1992}
M.~Filaseta and O.~Trifonov.
\newblock On gaps between squarefree numbers. {II}.
\newblock {\em J. London Math. Soc. (2)}, 45(2):215--221, 1992.

\bibitem{Ford18}
K.~Ford.
\newblock Extremal properties of product sets.
\newblock {\em Proc. Steklov Inst. Math.}, 303(1):220--226, 2018.
\newblock Published in Russian in Tr. Mat. Inst. Steklova {{\bf{3}}03} (2018), 239--245.

\bibitem{Gut}
A.~Gut.
\newblock {\em Probability: a graduate course}.
\newblock Springer Texts in Statistics. Springer, New York, second edition, 2013.

\bibitem{Harper}
A.~J. Harper.
\newblock On the limit distributions of some sums of a random multiplicative function.
\newblock {\em J. Reine Angew. Math.}, 678:95--124, 2013.

\bibitem{HarperLow}
A.~J. Harper.
\newblock Moments of random multiplicative functions, {I}: {L}ow moments, better than squareroot cancellation, and critical multiplicative chaos.
\newblock {\em Forum Math. Pi}, 8:e1, 95, 2020.

\bibitem{Helson}
H.~Helson.
\newblock Hankel forms.
\newblock {\em Studia Math.}, 198(1):79--84, 2010.

\bibitem{Hildebrand85}
A.~Hildebrand.
\newblock Integers free of large prime divisors in short intervals.
\newblock {\em Quart. J. Math. Oxford Ser. (2)}, 36(141):57--69, 1985.

\bibitem{Hooley}
C.~Hooley.
\newblock On the intervals between numbers that are sums of two squares. {III}.
\newblock {\em J. Reine Angew. Math.}, 267:207--218, 1974.

\bibitem{Hough}
B.~Hough.
\newblock Summation of a random multiplicative function on numbers having few prime factors.
\newblock {\em Math. Proc. Cambridge Philos. Soc.}, 150(2):193--214, 2011.

\bibitem{KSX}
O.~Klurman, I.~D. Shkredov, and M.~W. Xu.
\newblock On the random {C}howla conjecture.
\newblock {\em Geom. Funct. Anal.}, 33(3):749--777, 2023.

\bibitem{LTW2013}
Y.-K. Lau, G.~Tenenbaum, and J.~Wu.
\newblock On mean values of random multiplicative functions.
\newblock {\em Proc. Amer. Math. Soc.}, 141(2):409--420, 2013.

\bibitem{mastrostefano2020maximal}
D.~Mastrostefano.
\newblock On maximal product sets of random sets.
\newblock {\em J. Number Theory}, 224:13--40, 2021.

\bibitem{McLeish}
D.~L. McLeish.
\newblock Dependent central limit theorems and invariance principles.
\newblock {\em Ann. Probability}, 2:620--628, 1974.

\bibitem{MV77}
H.~L. Montgomery and R.~C. Vaughan.
\newblock Exponential sums with multiplicative coefficients.
\newblock {\em Invent. Math.}, 43(1):69--82, 1977.

\bibitem{PWX}
M.~{Pandey}, V.~Y. {Wang}, and M.~W. {Xu}.
\newblock {Partial sums of typical multiplicative functions over short moving intervals}.
\newblock {\em Algebra \& Number Theory}, to appear.
\newblock arXiv:2207.11758.

\bibitem{PS1990}
C.~Pomerance and A.~S\'{a}rk\"{o}zy.
\newblock On products of sequences of integers.
\newblock In {\em Number theory, {V}ol. {I} ({B}udapest, 1987)}, volume~51 of {\em Colloq. Math. Soc. J\'{a}nos Bolyai}, pages 447--463. North-Holland, Amsterdam, 1990.

\bibitem{Selberg}
A.~Selberg.
\newblock Note on a paper by {L}. {G}. {S}athe.
\newblock {\em J. Indian Math. Soc. (N.S.)}, 18:83--87, 1954.

\bibitem{Shiu1980}
P.~Shiu.
\newblock A {B}run-{T}itchmarsh theorem for multiplicative functions.
\newblock {\em J. Reine Angew. Math.}, 313:161--170, 1980.

\bibitem{xuzhou}
M.~W. Xu and Y.~Zhou.
\newblock On product sets of arithmetic progressions.
\newblock {\em Discrete Anal.}, pages Paper No. 10, 31, 2023.

\end{thebibliography}
\end{document}